\documentclass{article}
\usepackage[utf8]{inputenc}

\usepackage{algorithm}
\usepackage{algorithmic}
\usepackage{amsmath}
\usepackage{amssymb} 
\usepackage{amsfonts}
\usepackage{amsthm}
\usepackage{arydshln}
\usepackage{authblk}

\usepackage{bbold}
\usepackage{bm}
\usepackage{booktabs}
\usepackage{bookmark}
\usepackage{calc}              
\usepackage{cleveref}
\usepackage{enumitem}
\usepackage{extarrows}
\usepackage{float}
\usepackage{geometry}
\usepackage{graphicx}          
\usepackage{indentfirst}
\usepackage{latexsym}
\usepackage{mathtools}
\usepackage{multirow}          
\usepackage{subcaption}
\usepackage{tikz}              

\newtheorem{lemma}{Lemma}
\newtheorem{theorem}{Theorem}

\usepackage{authblk}

\title{Adaptive multiscale model reduction for linear elasticity equation in perforated domains}

\author[a,b]{Wei Xie}
\author[b]{Eric Chung}
\author[c]{Yin Yang\thanks{Corresponding author}}
\author[c]{Yunqing Huang}
\affil[a]{Hunan Key Laboratory for Computation and Simulation in Science and Engineering, National Center for Applied Mathematics in Hunan, Xiangtan University, Xiangtan 411105, Hunan, China}
\affil[b]{Department of Mathematics, The Chinese University of Hong Kong, Shatin, Hong Kong SAR, China}
\affil[c]{Hunan Research Center of the Basic Discipline Fundamental Algorithmic Theory and Novel Computational Methods, Key Laboratory of Intelligent Computing and Information Processing of Ministry of Education, Xiangtan University, Xiangtan 411105, Hunan, China}

\date{}

\begin{document}
\maketitle

\renewcommand{\thefootnote}{}
\footnotetext{E-mail addresses: xiew@smail.xtu.edu.cn (Wei Xie), 
tschung@math.cuhk.edu.hk (Eric Chung), 
yangyinxtu@xtu.edu.cn (Yin Yang), 
huangyq@xtu.edu.cn (Yunqing Huang)}
\renewcommand{\thefootnote}{\arabic{footnote}}

\maketitle

\begin{abstract}
In this paper, we develop a Constraint Energy Minimizing Generalized Multiscale Finite Element Method (CEM-GMsFEM) for solving linear elasticity problems in heterogeneous perforated domains. The presence of numerous perforations introduces multiple scales into the computational domain, making direct fine-grid simulations computationally expensive.
The proposed method follows the standard offline--online decomposition of CEM-GMsFEM. In the offline stage, local spectral problems are solved on coarse elements to construct auxiliary spaces, and localized energy-minimizing basis functions are then computed on oversampled regions to capture fine-scale geometric information induced by the perforations. In the online stage, residual-driven basis functions are constructed in enlarged coarse neighborhoods to incorporate source-term information and improve the accuracy of the multiscale approximation adaptively.
We establish convergence results for both the offline and online stages. In particular, we derive error estimates for the localized multiscale approximation and prove the convergence of the adaptive online enrichment algorithm. Moreover, we show that the oversampling regions used in the online stage can be determined locally, leading to a reduction in computational cost while maintaining convergence properties.
Numerical experiments on perforated media with different geometric configurations demonstrate the accuracy and efficiency of the proposed method.
\end{abstract}

\paragraph{keywords:}
Peforated domain, Multiscale model reduction, Linear elasticity

\section{Introduction} \label{sec:introduction}
Perforated domains arise in a wide range of applications, including porous media flow, composite materials, and subsurface modeling.
Linear elasticity models posed in perforated domains also serve as important components in the simulation of deformation processes in porous materials and fractured media \cite{mehmani2023multiscale, li2024multiscale}.
The presence of multiple perforations gives rise to complex multiscale features and effective heterogeneous behavior, which make direct numerical simulations computationally expensive due to the need for extremely fine meshes.
Therefore, developing efficient multiscale methods for perforated domains is of significant practical and theoretical interest.

One class of approaches seeks to derive effective macroscopic models that can be solved on coarse computational grids while preserving the essential influence of fine-scale heterogeneities.
These methods aim to reduce the computational complexity of direct simulations by replacing the original multiscale problem with an equivalent coarse-scale model.
Representative examples include homogenization methods \cite{allaire1991homogenizationCPAM, hornung1997homogenization, cao2003asymptotic}, which establish effective equations through asymptotic analysis,
the Heterogeneous Multiscale Method (HMM) \cite{weinan2003heterognous, ming2005analysis, henning2009heterogeneous}, which couples macroscopic solvers with localized microscopic simulations,
and multicontinuum homogenization methods \cite{efendiev2023multicontinuum, chung2024multicontinuum, xie2025multicontinuum, xie2026hierarchical, huang2026pinn}, which model multiple interacting continua to capture complex transport and mechanical behaviors across scales.

Another class of approaches constructs special multiscale basis functions.
Examples include the Generalized Finite Element Method (GFEM) \cite{babuska1983generalized, babuska1994special}
and the Multiscale Finite Element Method (MsFEM) \cite{hou1997multiscale, huang2001partition, le2014msfem, bris2019multiscale}.
The Generalized Multiscale Finite Element Method (GMsFEM) \cite{efendiev2013generalized, efendiev2013generalized, chung2016generalized, chung2017conservative, spiridonov2019generalized, xie2025multiscale, xie2025residual} enriches the multiscale approximation space by solving local spectral problems based on the MsFEM framework.
Similar spectral constructions have also been explored in the Multiscale Spectral Generalized Finite Element Method (MS-GFEM) \cite{ma2022novel, benezech2024scalable},
the Wavelet-Based Edge Multiscale Finite Element Method (WEMsFEM) \cite{fu2019edge, fu2021wavelet},
and Localized Subspace Iteration methods (LSI) \cite{guan2025localized, chung2026decoupling}.
The Localized Orthogonal Decomposition (LOD) method \cite{maalqvist2014localization, brown2016multiscale} decomposes the solution space into two orthogonal subspaces and establishes convergence with respect to the coarse mesh size.
The Constraint Energy Minimizing Generalized Multiscale Finite Element Method (CEM-GMsFEM) \cite{chung2018constraint, fu2020constraint, chung2021convergence, wang2023local, wang2024multiscale, xie2024cem, wang2025local, galvis2017overlapping} combines the main ideas of GMsFEM and LOD.
By constructing multiple multiscale basis functions in each coarse region, CEM-GMsFEM is able to capture complex local heterogeneities more accurately.
To further improve the approximation quality, online basis functions \cite{chung2018fast, fu2018constraint} are introduced to incorporate residual and source information, leading to faster convergence.

In this paper, we employ the CEM-GMsFEM framework to solve linear elasticity problems in perforated domains.
The computational procedure consists of offline and online stages.
In the first step, we solve local eigenvalue problems in each coarse block to obtain a collection of eigenfunctions that form the auxiliary space.
Subsequently, local energy minimization problems associated with the auxiliary space are solved to construct the offline multiscale basis functions.
In the online stage, residual-driven basis functions are constructed in oversampled coarse neighborhoods to incorporate source information.
The multiscale approximation space is then enriched by incorporating the newly constructed online basis functions.
In addition, we develop an adaptive enrichment algorithm that selects local subdomains for online basis construction based on local residual indicators.
We present convergence results for both the offline and online stages.
The flexibility of the oversampling layers for the offline basis functions was established in \cite{xie2024cem}.
In this work, we further prove the flexibility of the oversampling layers for the online basis functions, showing that the required oversampling regions can be determined using only local information, thereby reducing the computational cost.

The remainder of this paper is organized as follows.
In Section \ref{sec:preliminaries}, we introduce the preliminary concepts and notations.
The offline and online stages are presented in Sections \ref{sec:offlinebasis} and \ref{sec:onlinebasis}, respectively.
In Section \ref{sec:analysis}, we establish the offline convergence theorem and analyze the variability of oversampling layers in the online stage.
Numerical experiments are presented in Section \ref{sec:numericalexamples} to demonstrate the effectiveness of the proposed method.
Finally, conclusions are given in Section \ref{sec:conclusions}.

\section{Preliminaries} \label{sec:preliminaries}

The perforated domain $\Omega^{\epsilon}$ is obtained by removing from a bounded domain $\Omega \subset \mathbb{R}^d$ $(d=2,3)$ a collection of randomly distributed perforations.
The union of all perforations is denoted by $\mathcal{B}^{\epsilon}$ (see Figure \ref{fig:perforateddomain_example}).
We assume that the characteristic size of the perforations is much smaller than the diameter of the computational domain, namely, $0 < \epsilon \ll \mathrm{diam}(\Omega)$.

We consider the linear elasticity problem in the perforated domain $\Omega^{\epsilon}$:

\begin{equation}
\begin{cases}
\begin{aligned}
- \nabla \cdot \boldsymbol{\sigma}(u) &= f, &\text{in}~ \Omega^{\epsilon}, \\
\boldsymbol{\sigma}(u) &=  2\mu \boldsymbol{\varepsilon}(u) + \lambda \nabla \cdot u \mathcal{I}, 
&\text{in}~ \Omega^{\epsilon}, \\
\boldsymbol{\varepsilon}(u) &=  \frac{1}{2} \left( \nabla u + \nabla u^T \right), 
&\text{in}~ \Omega^{\epsilon}, \\
\boldsymbol{\sigma}(u) \cdot \Vec{\boldsymbol{n}} &= 0, &\text{on}~ \partial \mathcal{B}^{\epsilon} \cap \partial \Omega^{\epsilon}, \\
u &= 0, &\text{on}~ \partial \Omega \cap \partial \Omega^{\epsilon},
\end{aligned}
\end{cases}
\label{eq:pde}
\end{equation}
where $u$ denotes the displacement vector field, while $\boldsymbol{\sigma}$ and $\boldsymbol{\varepsilon}$ represent the stress and strain tensors, respectively.
The parameters $\lambda > 0$ and $\mu > 0$ are the Lam\'e coefficients, $\mathcal{I}$ denotes the identity tensor, and $\Vec{\boldsymbol{n}}$ is the outward unit normal vector.

Define the space
$V \coloneqq \{v\in [H^1(\Omega^{\epsilon})]^d : 
v|_{\partial \Omega^{\epsilon} \cap \partial \Omega}=0\}$, then the weak formulation of \eqref{eq:pde} is to find a solution $u \in V$ satisfied 
\begin{equation}
a(u, v) = (f, v), \quad \forall v \in V,
\label{eq:perforated_variation}
\end{equation}
where
\[
a(u, v) = \int_{\Omega^{\epsilon}} 
2\mu \boldsymbol{\varepsilon}(u) : \boldsymbol{\varepsilon}(v) + \lambda \nabla \cdot u \nabla \cdot v, \quad
(f, v) = \int_{\Omega^{\epsilon}} fv.
\]

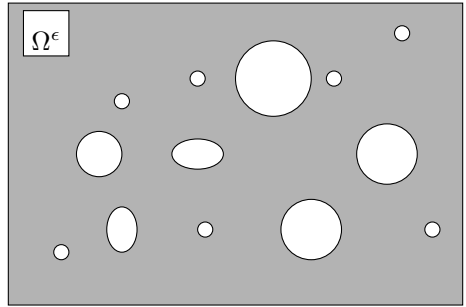
\begin{figure}[htbp]
\centering
\begin{tikzpicture}
\draw[fill=gray!60] (0, 0) rectangle (6, 4);
\draw[fill=white] (0.2, 3.3) rectangle (0.8, 3.9);
\draw[fill=white] (2.5, 2) ellipse (0.34 and 0.2);
\draw[fill=white] (1.5, 1) ellipse (0.2 and 0.3);
\draw[fill=white] (3.5, 3) circle (0.5);
\draw[fill=white] (1.2, 2) circle (0.3);
\draw[fill=white] (5, 2) circle (0.4);
\draw[fill=white] (4, 1) circle (0.4);
\draw[fill=white] (4.3, 3) circle (0.1);
\draw[fill=white] (1.5, 2.7) circle (0.1);
\draw[fill=white] (2.5, 3) circle (0.1);
\draw[fill=white] (2.6, 1) circle (0.1);
\draw[fill=white] (0.7, 0.7) circle (0.1);
\draw[fill=white] (5.6, 1) circle (0.1);
\draw[fill=white] (5.2, 3.6) circle (0.1);
\node at (0.5,3.5)   {$\Omega^{\epsilon}$};
\end{tikzpicture}
\caption{Illustration of a perforated domain.}
\label{fig:perforateddomain_example}
\end{figure}

We use $h$ to denote the fine mesh size that fully resolves the geometric heterogeneities induced by the perforations $\mathcal{B}^{\epsilon}$.
The fine-scale finite element space $V_h$ is defined on a conforming triangular mesh using piecewise linear finite elements ($P_1$ elements).
Unless otherwise specified, we omit the subscript $h$ for the fine-scale solution $u_h$ and the finite element space $V_h$ for notational simplicity.

Since the algebraic system associated with the fine mesh size $h$ is computationally expensive to solve, we employ the CEM-GMsFEM framework to construct multiscale basis functions that capture the essential local heterogeneities and enable efficient computations on a coarse grid.

For MsFEM and its generalized variants, the key ingredient is the construction of a suitable low-dimensional multiscale approximation space, denoted by $V_{\textnormal{ms}} \subset V$.
The corresponding multiscale formulation reads: find $u_{\textnormal{ms}} \in V_{\textnormal{ms}}$ such that
\begin{equation}
a(u_{\textnormal{ms}}, v) = (f, v), 
\quad \forall v \in V_{\textnormal{ms}}.
\label{eq:perforated_galerkin_ms}
\end{equation}

In Sections \ref{sec:offlinebasis} and \ref{sec:onlinebasis}, we introduce the construction of the offline and online multiscale spaces in the CEM-GMsFEM framework.
Before presenting these constructions, we first describe the coarse and fine meshes used in the multiscale discretization.
Let $\mathcal{T}^H$ be a coarse partition of the perforated domain $\Omega^{\epsilon}$.
The coarse mesh is used for the multiscale computation and does not need to resolve the perforation scale $\epsilon$.
Based on $\mathcal{T}^H$, we further refine each coarse element to obtain a conforming fine mesh $\mathcal{T}^h$.
We assume that $0 < h < \epsilon \ll H < \mathrm{diam}(\Omega)$.
We use $K$ to denote a coarse element and $\omega$ to denote a coarse neighborhood.
Let $N_c$ be the total number of coarse elements and $N_v$ the total number of interior coarse vertices of $\mathcal{T}^H$.
For the offline stage, the local problems for constructing offline basis functions are solved in the oversampled region $K_{i,k_i}$, which is obtained by enlarging the coarse block $K_i$ by $k_i$ coarse layers.
Similarly, for the online stage, the local problems for constructing online basis functions are solved in the oversampled region $\omega_{i,m_i}$, which is obtained by enlarging the coarse neighborhood $\omega_i$ by $m_i$ coarse layers.
In Figure \ref{pic:perforateddomain_mesh_example}, we illustrate the fine mesh, coarse mesh, and corresponding oversampling regions used in the multiscale construction.

\begin{figure}[htbp]
\centering
\begin{tikzpicture}
\centering
\node[anchor=south west,inner sep=0] at (0,0) {\includegraphics[height=3.6cm]{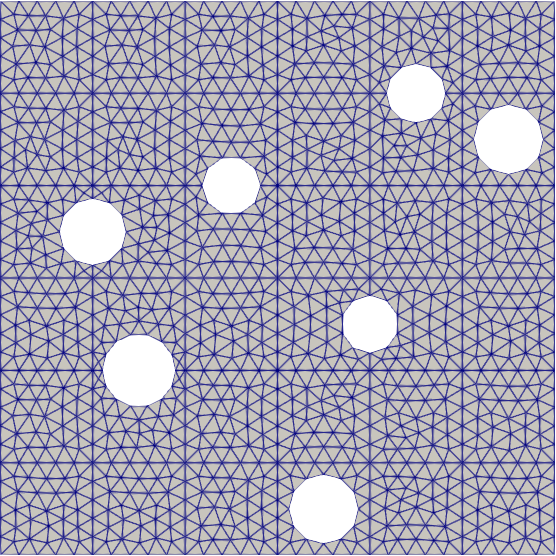}};
\node[anchor=south west,inner sep=0] at (4,0) {\includegraphics[height=3.6cm]{figs/6by6.png}};
\draw[red, line width=1.5pt] (0,0) grid[step=0.6] (3.6,3.6);
\draw[red, line width=1.5pt, shift={(4,0)}] (0,0) grid[step=0.6] (3.6,3.6);
\draw[yellow, line width=1.5pt] (1.2,1.2) rectangle (1.8,1.8);
\draw[green, line width=1.5pt] (0.6,0.6) rectangle (2.4,2.4);
\draw[blue, line width=1.5pt] (5.2,1.2) rectangle (6.4,2.4);
\draw[teal, line width=1.5pt] (4.6,0.6) rectangle (7,3);
\node at (-0.5,1.5)  {$K_i$};
\node at (-1.5,1.5)  {$K_{i,1}$};
\node at (8.2,1.8)  {$\omega_i$};
\node at (9.2,1.8)  {$\omega_{i,1}$};
\draw[line width=1pt] (-1.5,1.25) -- (0.6,0.6);
\draw[line width=1pt] (-1.5,1.75) -- (0.6,2.4);
\draw[line width=1pt] (-0.45,1.25) -- (1.2,1.2);
\draw[line width=1pt] (-0.42,1.75) -- (1.2,1.8);
\draw[line width=1pt] (6.4,1.2) -- (8.2,1.6);
\draw[line width=1pt] (6.4,2.4) -- (8.2,2); 
\draw[line width=1pt] (7,0.6) -- (9.2,1.6);
\draw[line width=1pt] (7,3) -- (9.2,2.1);
\end{tikzpicture} 
\caption{Illustration of the coarse grid, the fine-grid, the oversampling domain.}
\label{pic:perforateddomain_mesh_example}
\end{figure}

\section{Offline basis functions} \label{sec:offlinebasis}
In this section, we present the offline stage of the CEM-GMsFEM for problem \eqref{eq:pde}.
The construction of the offline multiscale space consists of two steps.
First, we construct the auxiliary multiscale space by solving a local spectral problem in each coarse block $K_i$ and selecting the eigenfunctions corresponding to the smallest $l_i$ eigenvalues.
Next, for each oversampled region $K_{i,k_i}$, obtained by enlarging $K_i$ by $k_i$ coarse-grid layers, we solve a constrained energy minimization problem to construct the offline multiscale basis functions.

\subsection{Auxiliary space}

We begin by constructing the auxiliary basis functions in each coarse block $K_i$.
To this end, we define the local bilinear forms $a_i(\cdot,\cdot)$ and $s_i(\cdot,\cdot)$ on $K_i$ by
\begin{equation}
a_i(v, w)=\int_{K_i} 2\mu \boldsymbol{\varepsilon}(v) : \boldsymbol{\varepsilon}(w) + \lambda \nabla \cdot v \nabla \cdot w,
\quad 
s_i(v,w)=\int_{K_i} \tilde{\kappa}vw.
\label{eq:as_innerproduct_define}
\end{equation}
Here, $\tilde{\kappa}= \sum_{j=1}^{N_v} |\nabla \chi_j|^2$, and
$\{\chi_j\}_{j=1}^{N_v}$ denotes a partition of unity on the coarse grid
\cite{babuvska1997partition}.
Typically, the partition of unity functions satisfy
$|\nabla \chi_j| = O(H^{-1})$ and
$0 \le \chi_j \le 1$.
For computational simplicity, we employ standard Lagrange basis functions.

For each coarse block $K_i$, we consider the following spectral problem:
\begin{equation}
a_i(\phi_j^i, v) = \lambda_j^i s_i(\phi_j^i, v),
\quad \forall v \in [H^1(K_i)]^d.
\label{eq:aux_weak}
\end{equation}

We can arrange the eigenvalues in ascending order such that
\[
\lambda_1^i \le \lambda_2^i \le \cdots \le \lambda_{l_i}^i \le \cdots,
\]
for each $i\in \{1,2,\cdots,N_c\}$.
By selecting the $l_i$ eigenfunctions corresponding to the first $l_i$ eigenvalues, we construct the local auxiliary space $V_{\textnormal{aux}}^i$.
Since the local auxiliary spaces have disjoint supports, we define the global auxiliary space as $V_{\textnormal{aux}} = \oplus_i V_{\textnormal{aux}}^i$.

It is worth noting that $\{\phi_j^i\}$ forms a set of unit orthogonal basis functions with respect to $\Vert \cdot \Vert_s^2 = \sum_{i=1}^{N_c} s_i(\cdot, \cdot)$. We can also define the global inner product as $s(u,v) = \sum_{i=1}^{N_c} s_i(u, v)$. Given a function $\phi_j^i \in V_{\text{aux}}$, we say that a function $\psi$ is $\phi_j^i\text{-orthogonal}$ if
\[
s(\psi, \phi_j^i)=1, \quad s(\psi, \phi_{j'}^{i'})=0, ~ \text{if } j'\ne j \text{ or } i'\ne i. 
\]
Additionally, we define a global operator $\pi : V \to V_{\text{aux}}$ as follows:
\[
\pi(u) = \sum_{i=1}^{N_c} \sum_{j=1}^{l_i} \frac{s_i(u, \phi_j^i)}{s_i(\phi_j^i, \phi_j^i)}\phi_j^i, \quad \forall u \in V. 
\]
Specifically, we can define the kernel of the projection $\pi$ as $\tilde{V} = \{v \in V ~|~ \pi(v)=0\}$.

\subsection{Offline space}

We now introduce two types of offline basis functions constructed in oversampled regions. For each coarse block $K_i$, let $K_{i,k_i}$ denote the oversampled region obtained by enlarging $K_i$ by $k_i$ coarse-grid layers (see Figure \ref{pic:perforateddomain_mesh_example}). For each auxiliary basis function $\phi_j^i \in V_{\text{aux}}^i$, we define $\psi_{j,\textnormal{ms}}^i$ as the solution of the following constrained energy minimization problem:
\begin{equation}
\psi_{j,\textnormal{ms}}^i = \mathrm{argmin}
\{ a(\psi,\psi)+s(\pi(\psi)-\phi_j^i, \pi(\psi)-\phi_j^i)~|~
\psi \in V_0(K_{i,k_i})\}.
\label{eq:r_mini}
\end{equation}
where $V_0(K_{i,k_i}) = \{v \in [H^1(K_{i,k_i})]^d ~|~
v=0 \text{ on } \partial K_{i,k_i} / (\partial K_{i,k_i}
\cap \partial \mathcal{B}^{\epsilon}) \}$.

The weak formulation of \eqref{eq:r_mini} is given by
\begin{equation}
a(\psi_{j,\textnormal{ms}}^i, v) + s(\pi(\psi_{j,\textnormal{ms}}^i), \pi(v)) =
s(\phi_j^i, \pi(v)), \quad
\forall v \in V_0(K_{i,k_i}).
\label{eq:r_variation}
\end{equation}

Next, we introduce another type of offline basis function. For each auxiliary basis function $\phi_j^i$, the corresponding multiscale basis function is defined as the solution of the following constrained energy minimization problem:
\begin{equation}
\psi_{j,\textnormal{ms}}^i = \mathrm{argmin}
\{a(\psi,\psi) ~ | ~\psi \in V_0(K_{i,k_i}), ~
\psi ~\textnormal{is}~ \phi_j^i\textnormal{-orthogonal} \}.
\label{eq:c_mini}
\end{equation}
The weak formulation of \eqref{eq:c_mini} can be rewritten using Lagrange multipliers following \cite{fu2018constraint}. Specifically, we seek $(\psi_{j,\textnormal{ms}}^i, w) \in V_0(K_{i,k_i}) \times V_{\textnormal{aux}}(K_{i,k_i})$ such that:
\begin{equation}
\begin{cases}
\begin{aligned}
a(\psi_{j,\textnormal{ms}}^i, v) + s(v,w) &= 0,
\quad \forall v\in V_0(K_{i,k_i}), \\
s(\psi_{j,\textnormal{ms}}^i-\phi_j^i, q) &= 0,
\quad \forall q\in V_{\text{aux}}(K_{i,k_i}),
\label{eq:c_variation}
\end{aligned}
\end{cases}
\end{equation}
where $V_{\textnormal{aux}}(K_{i,k_i})$ denotes the collection of all local auxiliary spaces associated with the coarse blocks $K_j \subset K_{i,k_i}$.
To distinguish between the two formulations, we refer to \eqref{eq:r_mini} as the relaxed version and \eqref{eq:c_mini} as the constrained version.
We are now ready to define the CEM multiscale space,
\[
V_{\textnormal{ms}} = \mathop{\mathrm{span}}_{i,j} \{ \psi_{j,\textnormal{ms}}^i \}.
\]

Indeed, the aforementioned multiscale basis functions can be viewed as localized approximations of the corresponding global basis functions, which are obtained by solving the problems over the entire computational domain. The relaxed version of the global basis function is defined by
\begin{equation}
\psi_{j,\textnormal{glo}}^i = \mathrm{argmin}
\{ a(\psi,\psi)+s(\pi(\psi)-\phi_j^i, \pi(\psi)-\phi_j^i)~|~
\psi \in V\},
\label{eq:r_mini_global}
\end{equation}
which is equivalent to the following variational formulation:
\begin{equation}
a(\psi_{j,\textnormal{glo}}^i, v) + s(\pi(\psi_{j,\textnormal{glo}}^i), \pi(v)) = 
s(\phi_j^i, \pi(v)), \quad 
\forall v \in V.
\label{eq:r_variation_global}
\end{equation}

The constrained version of the global basis function is defined by
\begin{equation}
\psi_{j,\textnormal{glo}}^i = \mathrm{argmin} \{ a(\psi,\psi) ~|~
\psi \in V, \psi \text{ is }
\phi_j^i \text{-orthogonal}\}.
\label{eq:c_mini_global}
\end{equation}
We can similarly derive the following equivalent variational formulation:
\begin{equation}
\begin{cases}
\begin{aligned}
a(\psi_{j,\textnormal{glo}}^i, v) + s(v,w) &= 0, \quad \forall v\in V, \\
s(\psi_{j,\textnormal{glo}}^i-\phi_j^i, q) &= 0, \quad \forall q\in V_{\text{aux}}.
\label{eq:c_variation_global}
\end{aligned}
\end{cases}
\end{equation}

Hence, we define the corresponding global multiscale space $V_{\textnormal{glo}}$ by
\[
V_{\text{glo}} = \mathop{\mathrm{span}}_{i,j} \{ \psi_{j,\textnormal{glo}}^i \}.
\]

Based on the global basis functions, we obtain the following global multiscale approximation: find $u_{\textnormal{glo}} \in V_{\textnormal{glo}}$ such that
\begin{equation}
a(u_{\text{glo}}, v) = (f, v), \quad \forall v \in V_{\text{glo}}.
\label{eq:perforated_galerkin_glo}
\end{equation}

The global space $V_{\textnormal{glo}}$ possesses two key properties that are fundamental to the CEM framework.
First, according to \cite{chung2018constraint}, the global basis functions $\psi_{j,\textnormal{glo}}^i$ satisfy a spatial decay property. Leveraging this property, we can construct local problems to approximate the global basis functions and thereby achieve an efficient multiscale model reduction.
Second, the global space $V_{\textnormal{glo}}$ is orthogonal to the kernel of the projection operator $\pi$. This orthogonality follows from the fact that $a(\psi_{j,\textnormal{glo}}^i, v)=0$ for any $v \in \tilde{V}$. Combining this property with the fact that $\mathrm{dim}(V_{\textnormal{glo}}) = \mathrm{dim}(V_{\textnormal{aux}})$, we obtain $V = V_{\textnormal{glo}} \oplus \tilde{V}$.

\section{Online basis functions} \label{sec:onlinebasis}

In this section, we present the online stage of CEM-GMsFEM for problem \eqref{eq:pde}. 
There are two key features that distinguish it from the offline stage.
First, the source term is incorporated into the construction of the online basis functions.
Second, the oversampling domains are tailored to individual coarse neighborhoods.

To begin with, we introduce a residual functional $r : V \rightarrow \mathbb{R}$. Let $u_{\textnormal{ms}}$ denote the multiscale solution of \eqref{eq:perforated_galerkin_ms}. The residual is then defined by
\[
r(v) = a(u_{\textnormal{ms}}, v) - (f, v), \quad \forall v \in V.
\]
Subsequently, we define the local residual $r_i : V \rightarrow \mathbb{R}$ associated with each coarse neighborhood $\omega_i$ by
\[
r_i(v) = r(\chi_i v), \quad \forall v \in V.
\]
The definition of the local residual $r_i$ is combined by the partition of unity function. In that so, we can divide the global residual into different local regions. If we use the coarse block, then the function in coarse edge is not easy to distinguish.

The construction of the online basis functions is based on the local residual $r_i$. Furthermore, we define the operator norm by
\[
\Vert r_i \Vert_{a_{\omega_i}^*} = \mathrm{sup}_{v\in V_0(\omega_i)} 
\frac{|r(v)|}{\Vert v \Vert_a}.
\]
The definition of the local residual $r_i$ is based on the partition of unity functions ${\chi_i}$.
This construction allows the global residual to be decomposed into local contributions associated with different coarse neighborhoods.
The use of coarse neighborhoods is essential in the online stage.
If the residual were localized on coarse blocks, the contributions of degrees of freedom located on coarse-grid interfaces would not be naturally assigned to a unique coarse block.
By contrast, the partition of unity functions provide a natural localization mechanism and enable the residual to be distributed consistently among neighboring coarse regions.

For a selected coarse neighborhood $\omega_i$, we enlarge it by $m_i$ coarse-grid layers to obtain an oversampled domain $\omega_{i,m_i}$, which serves as the support of the corresponding online basis function.
Specifically, the online basis function $\beta_{\textnormal{ms}}^i$ is defined as the solution of the following problem:
\begin{equation}
a(\beta_{\textnormal{ms}}^i, v) + s(\pi(\beta_{\textnormal{ms}}^i), \pi(v)) = r_i(v), \quad \forall v \in V_0(\omega_{i,m_i}),
\label{eq:beta_omegai}
\end{equation}
where \eqref{eq:beta_omegai} can be viewed as a localized approximation of the corresponding global online basis function $\beta_{\textnormal{glo}}^i \in V$, which is defined by
\begin{equation}
a(\beta_{\textnormal{glo}}^i, v) + s(\pi(\beta_{\textnormal{glo}}^i), \pi(v)) = r_i(v), \quad \forall v \in V.
\label{eq:beta_glo}
\end{equation}

Once the online basis functions have been constructed, a new multiscale solution is obtained by solving the problem in the enriched multiscale space. The adaptive enrichment algorithm is summarized in Algorithm \ref{alg:cem_online_adap}.
The adaptive component of the enrichment process is controlled by the adaptivity parameter $\theta$, which determines the coarse neighborhoods selected for online enrichment. Specifically, online basis functions are constructed only in regions associated with the largest local residual indicators. Consequently, computational resources are concentrated in the regions that contribute most significantly to the approximation error.
The algorithm reduces to a uniform enrichment strategy when $\theta = 1$, in which online basis functions are added in all coarse neighborhoods. In contrast, no online basis functions are added when $\theta = 0$.

\begin{algorithm}[!ht]
\caption{Residual-driven online adaptive algorithm.}
\begin{algorithmic}[1]
\STATE \textbf{Input:} Initial multiscale space $V_{\textnormal{ms}}^0$, initial multiscale solution $u_{\textnormal{ms}}^0$, maximum number of iterations \texttt{Iter}, adaptivity parameter $\theta$, and oversampling layers $\{m_i\}$.
\FOR {$n \in \{1, 2, \cdots, \texttt{Iter}\}$}
\FOR {$i \in \{1, 2, \cdots, N_v\}$}
\STATE Compute the local residual indicator
\[
\delta_i = \sup\limits_{v\in V_0(\omega_i)}
\frac{|r_i(v)|}{\Vert v \Vert_a}.
\]
\ENDFOR
\STATE Reindex the coarse neighborhoods such that
$\delta_1 \ge \delta_2 \ge \cdots$.
Select the smallest integer $q$ such that
\[
\sum_{i=q+1}^{N_v} \delta_i^2
\le
\theta
\sum_{i=1}^{N_v} \delta_i^2.
\]
\FOR {$i \in \{1, 2, \cdots, q\}$}
\STATE Compute $\beta_{\textnormal{ms}}^i$ by solving
\[
a(\beta_{\textnormal{ms}}^i, v)
+
s(\pi(\beta_{\textnormal{ms}}^i), \pi(v))
=
r_i(v),
\quad
\forall v \in V_0(\omega_{i,m_i}).
\]
\ENDFOR
\STATE Enrich the multiscale space:
\[
V_{\textnormal{ms}}^{n+1}
=
V_{\textnormal{ms}}^n
\oplus
\mathop{\mathrm{span}}_{1\le i \le q}
\{\beta_{\textnormal{ms}}^i\}.
\]
\STATE Solve \eqref{eq:perforated_galerkin_ms} in the enriched multiscale space $V_{\textnormal{ms}}^{n+1}$ to obtain $u_{\textnormal{ms}}^{n+1}$.
\ENDFOR
\STATE \textbf{Output:} Multiscale solution $u_{\textnormal{ms}}^n$ and multiscale space $V_{\textnormal{ms}}^n$.
\end{algorithmic}
\label{alg:cem_online_adap}
\end{algorithm}

\section{Convergence results} \label{sec:analysis}

In this section, we present the convergence analysis for both the offline and online stages. We first introduce several norms and constants that will be used throughout the analysis.

The $L^2$-norm and the energy norm are defined by
\[
\Vert u \Vert^2=\int_{\Omega^{\epsilon}} u^2,
\qquad
\Vert u \Vert_a^2=\int_{\Omega^{\epsilon}}
2\mu \boldsymbol{\varepsilon}(u):\boldsymbol{\varepsilon}(u)
+\lambda (\nabla\cdot u)^2.
\]
We also define
\[
C_0=\sup_{v\in V}
\frac{\Vert \pi(v)\Vert_s^2}{\Vert v\Vert_a^2},
\]
which measures the stability of the projection operator $\pi$ with respect to the energy norm.

For any subdomain $\omega \subset \Omega^{\epsilon}$, let
\[
\Lambda_{\omega} =
\min_{K_i\cap\omega\neq\emptyset}
\lambda_{l_i+1}^i,
\qquad
\Gamma_{\omega} =
\max_{K_i\cap\omega\neq\emptyset}
\lambda_{l_i}^i.
\]
These quantities characterize the local spectral information of the auxiliary eigenvalue problems.

The convergence analysis of the offline multiscale space follows closely the framework developed in \cite{chung2018constraint, xie2024cem}. Since the main arguments are similar, we omit the technical details and only state the resulting error estimate. The following theorem unifies the convergence results for both the relaxed multiscale basis functions defined by \eqref{eq:r_mini} and the constraint multiscale basis functions defined by \eqref{eq:c_mini}.

\begin{theorem}
Let $u$ be the solution of \eqref{eq:perforated_variation} and let
$u_{\textnormal{ms}}$ be the offline multiscale solution of \eqref{eq:perforated_galerkin_ms}. Assume that the multiscale space $V_{\textnormal{ms}}$ is constructed using either the relaxed basis functions \eqref{eq:r_mini} or the constraint basis functions \eqref{eq:c_mini}. Then,
\[
\Vert u-u_{\textnormal{ms}}\Vert_a
\le
C_1 \Lambda_{\Omega^{\epsilon}}^{-\frac{1}{2}}
\Vert \tilde{\kappa}^{-\frac{1}{2}} f\Vert
+
C_2 E \Vert u_{\textnormal{glo}}\Vert_s,
\]
where $C_1$ and $C_2$ are positive constants, and $E$ denotes the exponential decay factor associated with the oversampling layers. The precise form of $E$ depends on whether the relaxed or constraint basis functions are employed.

Furthermore, if each oversampling parameter $k_i$ is sufficiently large and $\{\chi_i\}_{i=1}^{N_v}$ is a set of bilinear partition of unity functions, then
\[
\Vert u-u_{\textnormal{ms}} \Vert_a
\le
C H \Lambda_{\Omega^{\epsilon}}^{-\frac{1}{2}}
\Vert f\Vert.
\]
\end{theorem}

Before prove the adaptive online convergence, we will review some lemma that in \cite{xie2024cem, chung2018fast}.
For any $v \in V$, we define an enlarged support, denoted by
$\underline{\mathrm{supp}}(v)$, as
\[
\underline{\mathrm{supp}}(v) = \bigcup_{v|_{K_i} \ne 0} K_i.
\]

\begin{lemma} \label{lemma:vaux_v}
For any $v_{\textnormal{aux}} \in V_{\textnormal{aux}}$, there exists a function $v \in V$ such that
\[
\pi(v)=v_{\textnormal{aux}}, \quad
\Vert v \Vert_a^2 \le D_{\underline{\mathrm{supp}}(v)} \Vert v_{\textnormal{aux}} \Vert_s^2, \quad
\underline{\mathrm{supp}}(v) \subset \underline{\mathrm{supp}}(v_{\textnormal{aux}}),
\]
where $D_{\underline{\mathrm{supp}}(v)}=C_{\mathcal{T}}(1+\Gamma_{\underline{\mathrm{supp}}(v)})$, and $C_{\mathcal{T}}$ denotes the square of the maximum number of vertices among all coarse elements.
\end{lemma}

\begin{lemma} \label{lemma:glolocal_r}
Let $\phi_j^i$ be a given basis function in the auxiliary space $V_{\textnormal{aux}}^i$, and let $\psi_{j,\textnormal{glo}}^i$ and $\psi_{j,\textnormal{ms}}^i$ denote the corresponding global and localized multiscale basis functions defined by \eqref{eq:r_mini_global} and \eqref{eq:r_mini}, respectively. Then, the following estimate holds:
\[
\Vert \psi_{j,\textnormal{glo}}^i - \psi_{j,\textnormal{ms}}^i \Vert_a^2 + 
\Vert \pi(\psi_{j,\textnormal{glo}}^i - \psi_{j,\textnormal{ms}}^i) \Vert_s^2 \le 
E_i \left( \Vert \psi_{j,\textnormal{glo}}^i \Vert_a^2 + \Vert \pi(\psi_{j,\textnormal{glo}}^i) \Vert_s^2 \right),
\]
where $E_i$ is the exponential decay factor.
\end{lemma}

\begin{lemma} \label{lemma:psi_gloms_recu}
With the same notation in Lemma \ref{lemma:glolocal_r}, we have
\[
\begin{aligned}
& \Vert \sum_{i=1}^{N_c} \sum_{j=1}^{l_i} c_j^i (\psi_{j,\textnormal{glo}}^i - \psi_{j,\textnormal{ms}}^i) \Vert_a^2 + 
\Vert \sum_{i=1}^{N_c} \sum_{j=1}^{l_i} c_j^i \pi(\psi_{j,\textnormal{glo}}^i - \psi_{j,\textnormal{ms}}^i) \Vert_s^2 \\
\le & C (1+\Lambda_{\Omega^{\epsilon}}^{-1}) (k+1)^d 
\sum_{i=1}^{N_c} \left( \Vert \sum_{j=1}^{l_i} c_j^i (\psi_{j,\textnormal{glo}}^i - \psi_{j,\textnormal{ms}}^i) \Vert_a^2 + 
\Vert \sum_{j=1}^{l_i} c_j^i \pi(\psi_{j,\textnormal{glo}}^i - \psi_{j,\textnormal{ms}}^i) \Vert_s^2 \right).    
\end{aligned}
\]
\end{lemma}

Similar to the decay properties established for the global offline multiscale basis functions in \cite{xie2024cem}, one can show that the global online basis functions $\beta_{\textnormal{glo}}^i$ also exhibit exponential decay. The corresponding result is stated below. The proof is omitted since it follows the same arguments as those in \cite{xie2024cem}.

\begin{lemma} \label{lemma:beta_decay}
Let $\omega_{i,m_i}$ be the oversampled region obtained by enlarging the coarse neighborhood $\omega_i$ by $m_i$ coarse layers, where $m_i \ge 2$. Let $\beta_{\textnormal{glo}}^i$ and $\beta_{\textnormal{ms}}^i$ denote the global and localized online basis functions defined by \eqref{eq:beta_glo} and \eqref{eq:beta_omegai}, respectively. Then,
\[
\Vert \beta_{\textnormal{glo}}^i - \beta_{\textnormal{ms}}^2 \Vert_a^2 + 
\Vert \pi(\beta_{\textnormal{glo}}^i - \beta_{\textnormal{ms}}^i )\Vert_s^2 \le 
E_i (\Vert \beta_{\textnormal{glo}}^i \Vert_a^2 + \Vert \pi(\beta_{\textnormal{glo}}^i )\Vert_s^2),
\]
where $E_i=3(1+\Lambda_{\Omega^{\epsilon}\setminus \omega_{i,m_i-1}}^{-1}) \left(1+(2(1+\Lambda_{\omega_{i,m_i-1}}^{-\frac{1}{2}}))^{-1}\right)^{1-m_i}$ is the exponential decay factor.
\end{lemma}

Next, we present the convergence result for the online stage. The following theorem establishes a recursive error estimate for the adaptive online enrichment algorithm and identifies the main factors that govern its convergence rate.

\begin{theorem}\label{thm:cem:on}
Let $u$ be the fine-grid solution of \eqref{eq:pde}, and let
$\{u_{\textnormal{ms}}^{n}\}$ be the sequence of multiscale solutions generated by the online enrichment Algorithm \ref{alg:cem_online_adap}. Then, for the $(n+1)$-th online iteration, the following recursive error estimate holds:
\[
\Vert u-u_{\textnormal{ms}}^{n+1}\Vert_a^2
\le
3(1+\Lambda_{\Omega^{\epsilon}}^{-1})(1+D_{\Omega^{\epsilon}})
\bigl(
C_1 E + C_2(1-\theta)
\bigr)
\Vert u_h-u_{\textnormal{ms}}^{n}\Vert_a^2.
\]
Here, $E$ represents an upper bound for the exponential decay factors associated with the offline and online multiscale basis functions. The quantities $C_1$ and $C_2$ are positive constants, while $\theta$ denotes the adaptivity parameter in Algorithm \ref{alg:cem_online_adap}.
\end{theorem}
\begin{proof}
For the multiscale solution $u_{\textnormal{ms}}^n$ obtained at the $n$-th iteration, we have
\begin{equation}
a(u_{\textnormal{ms}}^n-u_h, v) 
= \sum_{i=1}^{N_v} (a(u_{\textnormal{ms}}^n, \chi_i v) - (f, \chi_i v)) \\
= \sum_{i=1}^{N_v} r_i^n(v).
\label{eq:uhu_r}
\end{equation}
By the definition of the global online basis functions given in \eqref{eq:beta_glo}, we obtain
\begin{equation}
a(\sum_{i=1}^{N_v} \beta_{\textnormal{glo}}^i, v) + s(\pi(\sum_{i=1}^{N_v} \beta_{\textnormal{glo}}^i), \pi(v)) = \sum_{i=1}^{N_v} r_i^n(v).
\label{eq:beta_glo_sum}
\end{equation}
Combining \eqref{eq:uhu_r} with \eqref{eq:beta_glo_sum}, we obtain
\begin{equation}
a(u-u_{\textnormal{ms}}^n+\sum_{i=1}^{N_v} \beta_{\textnormal{glo}}^i, v) = 
s(-\pi(\sum_{i=1}^{N_v} \beta_{\textnormal{glo}}^i), \pi(v)), \quad
\forall v \in V_h.
\label{eq:uuh_beta}
\end{equation}
Therefore, for any $v\in\tilde{V}$, we have
\[
a(u-u_{\textnormal{ms}}^n+\sum_{i=1}^{N_v} \beta_{\textnormal{glo}}^i, v) = 0.
\]
Using the decomposition $V_h = V_{\textnormal{glo}} \oplus \tilde{V}$, we obtain
\[
u-u_{\textnormal{ms}}^n+\sum_{i=1}^{N_v} \beta_{\textnormal{glo}}^i \in 
\tilde{V}^{\perp} = V_{\textnormal{glo}}.
\]
Therefore,
\[
u-u_{\textnormal{ms}}^n+\sum_{i=1}^{N_v} \beta_{\textnormal{glo}}^i = 
\sum_{i=1}^{N_c} \sum_{j=1}^{l_i} c_j^i \psi_{j,\textnormal{glo}}^i,
\]
where $c_j^i$ denotes the coefficient associated with the basis function $\psi_{j,\textnormal{glo}}^i$.
For the adaptive enrichment Algorithm \ref{alg:cem_online_adap}, let $I=\{1,\ldots,q\}$ denote the index set of coarse neighborhoods selected for online enrichment.
\begin{equation}
\Vert u_h-u_{\textnormal{ms}}^{n+1} \Vert_a \le 
\Vert u_h-w \Vert_a, \quad
\forall w \in V_{\textnormal{ms}}^{n+1}.
\label{eq:galerkin_uw}
\end{equation}
We define
\[
w = u_{\textnormal{ms}}^n - \sum_{i\in I} 
\beta_{\textnormal{ms}}^i + 
\sum_{i=1}^{N_c} \sum_{j=1}^{l_i} c_j^i \psi_{j,\textnormal{ms}}^i.
\]
Clearly, $w \in V_{\textnormal{ms}}^{n+1}$, we have
\begin{equation}
\begin{aligned}
&\Vert u-u_{\textnormal{ms}}^{n+1} \Vert_a^2
\le \Vert u-u_{\textnormal{ms}}^n + \sum_{i\in I} \beta_{\textnormal{ms}}^i 
- \sum_{i=1}^{N_c} \sum_{j=1}^{l_i} c_j^i \psi_{j,\textnormal{ms}}^i \Vert_a^2 \\
&= \Vert \sum_{i\in I} (\beta_{\textnormal{ms}}^i-\beta_{\textnormal{glo}}^i) 
-\sum_{i\notin I} \beta_{\textnormal{glo}}^i + 
\sum_{i=1}^{N_c} \sum_{j=1}^{l_i} c_j^i (\psi_{j,\textnormal{glo}}^i-\psi_{j,\textnormal{ms}}^i) \Vert_a^2 \\
&\le 3\left(
\Vert \sum_{i\in I} (\beta_{\textnormal{ms}}^i-\beta_{\textnormal{glo}}^i)  \Vert_a^2 +
\Vert \sum_{i\notin I} \beta_{\textnormal{glo}}^i  \Vert_a^2 + 
\Vert \sum_{i=1}^{N_c} \sum_{j=1}^{l_i} c_j^i (\psi_{j,\textnormal{glo}}^i-\psi_{j,\textnormal{ms}}^i) \Vert_a^2
\right).
\end{aligned}
\label{eq:u_uhnew_w}
\end{equation}

We now estimate the three terms on the right-hand side of \eqref{eq:u_uhnew_w} separately.

\paragraph{Step 1.}
We first estimate $\Vert \sum_{i\in I} (\beta_{\textnormal{ms}}^i-\beta_{\textnormal{glo}}^i)  \Vert_a^2$.
By the definition of the global online basis functions $\beta_{\textnormal{glo}}^i$ in \eqref{eq:beta_glo}, we have
\begin{equation}
\begin{aligned}
\Vert \beta_{\textnormal{glo}}^i \Vert_a^2 + \Vert \pi(\beta_{\textnormal{glo}}^i) \Vert_s^2 &=
a(u-u_{\textnormal{ms}}^n, \chi_i \beta_{\textnormal{glo}}^i) \\
&\le \Vert u-u_{\textnormal{ms}}^n \Vert_{a(\omega_i)} 
\Vert \chi_i \beta_{\textnormal{glo}}^i \Vert_{a(\omega_i)} \\
&\le \sqrt{2} \Vert u-u_{\textnormal{ms}}^n \Vert_{a(\omega_i)} 
(\Vert \beta_{\textnormal{glo}}^i \Vert_{a(\omega_i)}^2 + \Vert \beta_{\textnormal{glo}}^i \Vert_{s(\omega_i)}^2)^{\frac{1}{2}},
\end{aligned}
\label{eq:beta_u_gloms}
\end{equation}
and
\begin{equation}
\begin{aligned}
\Vert \beta_{\textnormal{glo}}^i \Vert_{s(\omega_i)}^2 
&\le \Vert \pi(\beta_{\textnormal{glo}}^i) \Vert_{s(\omega_i)}^2
+ \Vert (I-\pi)(\beta_{\textnormal{glo}}^i) \Vert_{s(\omega_i)}^2 \\
&\le \Vert \pi(\beta_{\textnormal{glo}}^i) \Vert_{s(\omega_i)}^2
+ \Lambda_{\omega_i}^{-1} \Vert \beta_{\textnormal{glo}}^i \Vert_{a(\omega_i)}^2.
\end{aligned}
\label{eq:beta_ssaw}
\end{equation}
Using \eqref{eq:beta_u_gloms}, \eqref{eq:beta_ssaw} and Lemma \ref{lemma:beta_decay}, we have
\[
\begin{aligned}
\sum_{i=1}^{N_v} (\Vert \beta_{\textnormal{glo}}^i - \beta_{\textnormal{ms}}^i \Vert_a^2 + 
\Vert \pi(\beta_{\textnormal{glo}}^i - \beta_{\textnormal{ms}}^i) \Vert_s^2) 
&\le 2 \sum_{i=1}^{N_v} (1+\Lambda_{\omega_i}^{-1}) E_i \Vert u-u_{\textnormal{ms}}^n \Vert_{a(\omega_i)}^2 \\
&\le 8 E (1+\Lambda_{\Omega^{\epsilon}}^{-1}) \Vert u - u_{\textnormal{ms}}^m \Vert_a^2.
\end{aligned}
\]

\paragraph{Step 2.}
We next estimate the second term $\Vert \sum_{i\notin I} \beta_{\textnormal{glo}}^i  \Vert_a^2$, and define $p \coloneqq \sum_{i\notin I} \beta_{\textnormal{glo}}^i$. 
By the definition of the global online basis functions in \eqref{eq:beta_glo}, we have
\[
\begin{aligned}
\Vert p \Vert_a^2 + \Vert \pi(p) \Vert_s^2 &= r^{n+1}(\sum_{i\notin I} \chi_i p) 
\le \sum_{i\notin I} \left( 
\mathop{\mathrm{sup}}_{v \in V_0(\omega_i)} \frac{r^{n+1}(v)}{\Vert v \Vert_a} 
\right) \Vert \chi_i p \Vert_a \\
&\le \sqrt{2} \sum_{i\notin I} \Vert r_i^{n+1} \Vert_{a^*} 
\left(
\Vert p \Vert_{a(\omega_i)}^2 + \Vert p \Vert_{s(\omega_i)}^2
\right)^{\frac{1}{2}} \\
&\le \sqrt{2} (1+\Lambda_{\Omega^{\epsilon}}^{-1})^{\frac{1}{2}} 
\sum_{i\notin I} \Vert r_i^{n+1} \Vert_{a^*} 
\left(
\Vert p \Vert_{a(\omega_i)}^2 + \Vert \pi(p) \Vert_{s(\omega_i)}^2
\right)^{\frac{1}{2}}, \\
\end{aligned}
\]
Therefore,
\[
\Vert \sum_{i\notin I} \beta_{\textnormal{glo}}^i \Vert_a^2 = 
\Vert p \Vert_a^2 \le 
8 (1+\Lambda_{\Omega^{\epsilon}}^{-1}) \sum_{i\notin I} \Vert r_i^{n+1} \Vert_{a^*}^2 \le 
8 (1+\Lambda_{\Omega^{\epsilon}}^{-1}) (1-\theta) \sum_{i=1}^{N_v} \Vert r_i^{n+1} \Vert_{a^*}^2,
\]
where $\theta$ is the adaptivity parameter in Algorithm \ref{alg:cem_online_adap}.
By the definition of $r_i^{n+1}$, for any $v\in V_0(\omega_i)$, we have
\[
r_i^{n+1}(v) = a(u-u_{\textnormal{ms}}^n, v) \le \Vert u-u_{\textnormal{ms}}^n \Vert_{a(\omega_i)} \Vert v \Vert_a.
\]
Using this estimate, we obtain
\[
\sum_{i=1}^{N_v} \Vert r_i^{n+1} \Vert_{a^*}^2 \le \sum_{i=1}^{N_v} \Vert u-u_{\textnormal{ms}}^n \Vert_{a(\omega_i)} \le 4 \Vert u-u_{\textnormal{ms}}^n \Vert_a^2.
\]
Therefore,
\[
\Vert \sum_{i\notin I} \beta_{\textnormal{glo}}^i  \Vert_a^2 \le 
32 (1+\Lambda_{\Omega^{\epsilon}}^{-1}) (1-\theta) \Vert u-u_{\textnormal{ms}}^n \Vert_a^2.
\]

\paragraph{Step 3.}
It remains to estimate the last term in \eqref{eq:u_uhnew_w}.
We will show that this term can be bounded by the error from the previous iteration, $\Vert u-u_{\textnormal{ms}}^n \Vert_a^2$. 
Combining Lemmas \ref{lemma:glolocal_r} and \ref{lemma:psi_gloms_recu}, we obtain
\begin{equation}
\begin{aligned}
\Vert \sum_{i=1}^{N_c} \sum_{j=1}^{l_i} c_j^i (\psi_{j,\textnormal{glo}}^i-\psi_{j,\textnormal{ms}}^i) \Vert_a^2 
\le
E (1+\Lambda_{\Omega^{\epsilon}}^{-1}) (m+1)^d  \sum_{i=1}^{N_c} 
\sum_{j=1}^{l_i} c_j^i ( \Vert \psi_{j,\textnormal{glo}}^i \Vert_a^2 + \Vert \pi(\psi_{j,\textnormal{glo}}^i) \Vert_s^2 ).
\end{aligned}
\label{eq:cij_psi_ms_glo}
\end{equation}

Define
$\eta \coloneqq
\sum_{i=1}^{N_c} \sum_{j=1}^{l_i} c_j^i \psi_{j,\textnormal{glo}}^i
= u-u_{\textnormal{ms}}^n+\sum_{i=1}^{N_v} \beta_{\textnormal{glo}}^i$, we have
\begin{equation}
\begin{aligned}
a(\eta, v) + s(\pi(\eta), \pi(v)) &= 
a(\sum_{i=1}^{N_c} \sum_{j=1}^{l_i} c_j^i \psi_{j,\textnormal{glo}}^i, v) + 
s(\sum_{i=1}^{N_c} \sum_{j=1}^{l_i} c_j^i \pi(\psi_{j,\textnormal{glo}}^i), \pi(v)) \\
&= \sum_{i=1}^{N_c} \sum_{j=1}^{l_i} c_j^i s(\phi_j^i, \pi(v)),
\end{aligned}
\label{eq:eta_phisum}
\end{equation}
Let $v_{\textnormal{aux}}^i = \sum_{j=1}^{l_i} c_j^i \phi_j^i \in V_{\textnormal{aux}}^i$. By Lemma \ref{lemma:vaux_v}, there exists a function $q^i \in V_0(K_i)$ such that $\pi(q^i) = v_{\textnormal{aux}}^i$ and
\[
\Vert q^i \Vert_a^2 \le D_{K_i} \Vert v_{\textnormal{aux}}^i \Vert_s^2.
\]
Taking $v = q^i$ in \eqref{eq:eta_phisum}, we obtain
\[
\begin{aligned}
\Vert v_{\textnormal{aux}}^i \Vert_s^2 &= a(\eta, q^i) + s(\pi(\eta), \pi(q^i)) \\
&\le (\Vert \eta \Vert_{a(K_i)}^2 + \Vert \pi(\eta) \Vert_{s(K_i)}^2)^{\frac{1}{2}}
(\Vert q^i \Vert_{a(K_i)}^2 + \Vert \pi(q^i) \Vert_{s(K_i)}^2)^{\frac{1}{2}} \\
&\le (\Vert \eta \Vert_{a(K_i)}^2 + \Vert \pi(\eta) \Vert_{s(K_i)}^2)^{\frac{1}{2}}
((1+D_{K_i})\Vert v_{\textnormal{aux}}^i \Vert_{s(K_i)}^2)^{\frac{1}{2}}. \\
\end{aligned}
\]
Since the eigenfunctions $\phi_j^i$ are mutually orthogonal, we have
\[
\sum_{i=1}^{N_c} \sum_{j=1}^{l_i} (c_j^i)^2 = \sum_{i=1}^{N_c} \Vert \sum_{j=1}^{l_i} c_j^i \psi_{j,\textnormal{glo}}^i \Vert_s^2 
\le (1+D_{\Omega^{\epsilon}}) (\Vert \eta \Vert_a^2 + \Vert \pi(\eta)\Vert_s^2)
\le (1+D_{\Omega^{\epsilon}}) (C_0+1) \Vert \eta \Vert_a^2.
\]
Recalling the definition of $\eta$, we obtain
\[
\begin{aligned}
\sum_{i=1}^{N_c} \sum_{j=1}^{l_i} (c_j^i)^2 &\le 
(1+D_{\Omega^{\epsilon}}) (C_0+1) \Vert u - u_{\textnormal{ms}}^n + \sum_{i=1}^{N_v} \beta_{\textnormal{glo}}^i \Vert_a^2 \\
&\le 2(1+D_{\Omega^{\epsilon}}) (C_0+1) (\Vert u - u_{\textnormal{ms}}^n \Vert_a^2 + 
\Vert \sum_{i=1}^{N_v} \beta_{\textnormal{glo}}^i \Vert_a^2) \\
&\le 4(1+D_{\Omega^{\epsilon}}) (C_0+1) \Vert u - u_{\textnormal{ms}}^n \Vert_a^2,
\end{aligned}
\]
Taking $v=\sum_{i=1}^{N_v} \beta_{\textnormal{glo}}^i$ in \eqref{eq:uuh_beta}, we obtain the desired inequality.

From the weak formulation \eqref{eq:r_variation_global} of the relaxed global basis functions, we have
\begin{equation}
\Vert \psi_{j,\textnormal{glo}}^i \Vert_a^2 + \Vert \pi(\psi_{j,\textnormal{glo}}^i) \Vert_s^2
\le \Vert \phi_j^i \Vert_s^2 = 1.
\label{eq:phipsiineq}
\end{equation}

Finally, combining \eqref{eq:cij_psi_ms_glo} and \eqref{eq:phipsiineq}, we obtain
\[
\Vert \sum_{i=1}^{N_c} \sum_{j=1}^{l_i} c_j^i 
(\psi_{j,\textnormal{glo}}^i - \psi_{j,\textnormal{ms}}^i) \Vert_a^2 \le 
4 E (C_0+1) (1+\Lambda_{\Omega^{\epsilon}}^{-1}) (m+1)^d (1+D_{\Omega^{\epsilon}}) \Vert u - u_{\textnormal{ms}}^n \Vert_a^2.
\]

The proof follows by combining the estimates from Steps 1--3.
\end{proof}

At this point, the convergence analysis of the online stage is complete. 
Theorem \ref{thm:cem:on} shows that the error reduction achieved by the online iteration is governed by three main factors.
First, the offline multiscale space must possess sufficient approximation capability. 
In particular, the local multiscale basis functions should be constructed with a sufficiently large number of oversampling layers and a sufficient number of offline basis functions.
Second, the convergence is influenced by the exponential decay rate of the global online basis functions. 
This indicates that the construction of online basis functions also requires sufficiently large oversampling regions in order to accurately capture the global information.
Third, the error reduction depends on the number of online basis functions added at each enrichment step. 
This quantity is controlled by the adaptivity parameter $\theta$ in Algorithm \ref{alg:cem_online_adap}.
When $\theta = 1$, corresponding to uniform enrichment, this factor disappears from the error estimate, and the associated error contribution is eliminated.

\section{Numerical examples} \label{sec:numericalexamples}

In this section, we present two perforated media to demonstrate the performance of the proposed method. The first medium contains a relatively small number of perforations, whereas the second one consists of a much denser distribution of perforations; see Figure \ref{fig:media}.

For both test cases, the Lam\'e coefficients are set to $\lambda=1$ and $\mu=0.5$, corresponding to a homogeneous isotropic material. The fine-grid mesh is generated using triangular elements with mesh size $h=1/200$.

To evaluate the accuracy of CEM-GMsFEM, we consider the relative $L^2$ error and the relative energy error:
\[
e_{L^2} \coloneqq \frac{\Vert u_h-u_{\textnormal{ms}} \Vert}{\Vert u_h \Vert} , \quad
e_{H^1} \coloneqq \frac{\Vert u_h-u_{\textnormal{ms}} \Vert_{a}}{\Vert u_h \Vert_{a}}.
\]
Here, $u_h$ denotes the reference fine-grid solution and $u_{\mathrm{ms}}$ denotes the corresponding multiscale solution.

\begin{figure}[htbp]
\centering
\includegraphics[scale=0.4]{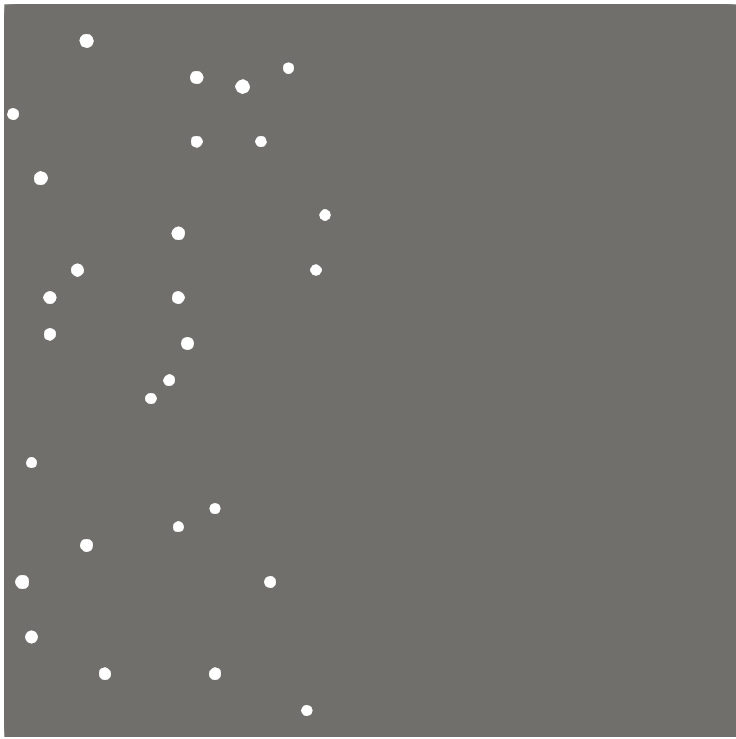}
\includegraphics[scale=0.4]{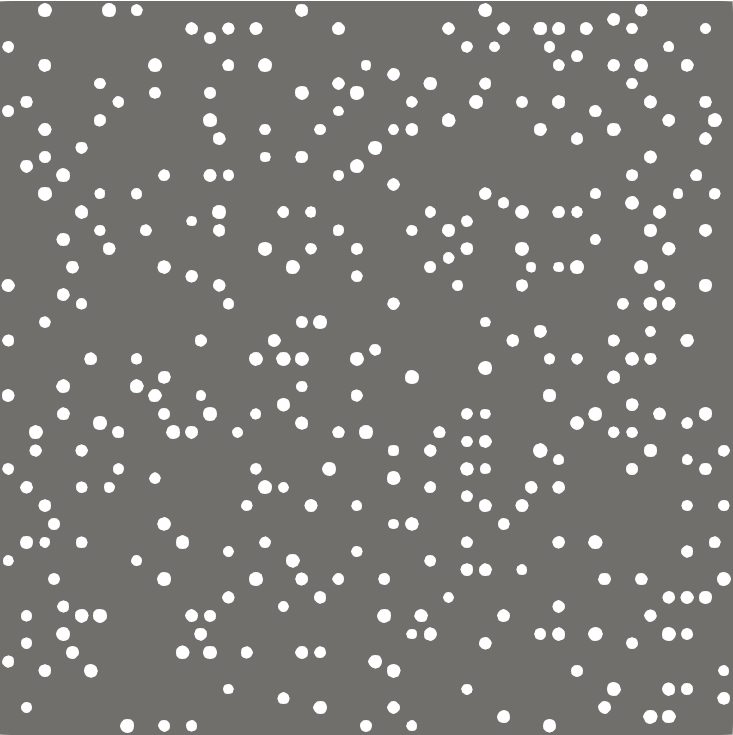}
\caption{Left: first media for Case 1. Right: second media for Case 2.}
\label{fig:media}
\end{figure}

\subsection{Case 1}

In this example, we investigate the performance of CEM-GMsFEM in the first perforated medium shown in Figure \ref{fig:media}. The corresponding source term is displayed in Figure \ref{fig:source_case1}.

For the offline stage, we employ four auxiliary basis functions in each coarse block and consequently construct four multiscale basis functions per coarse block. The convergence results for the constraint and relaxed formulations are shown in Figures \ref{fig:case1_error_c} and \ref{fig:case1_error_r}, respectively. As expected from the theoretical analysis, the approximation error decreases as the coarse mesh size $H$ is refined, provided that a sufficient number of oversampling layers is used. The numerical results reported in Table \ref{tab:case1_error} further confirm this observation. Increasing the number of oversampling layers improves the accuracy; however, the error eventually becomes dominated by the coarse-grid discretization error. A comparison between the reference solution and the multiscale solution is presented in Figure \ref{fig:uuhplot_case1}, demonstrating that the offline multiscale space provides an accurate approximation of the fine-scale solution.

For the online stage, we again use four basis functions in each coarse block and fix the coarse mesh size to $H=1/20$. Unless otherwise specified, three oversampling layers ($m=3$) are employed for each coarse neighborhood. In the adaptive enrichment procedure, we set $\theta=0.1$ and perform five online enrichment iterations, i.e., $\texttt{Iter}=5$. The corresponding results are reported in Table \ref{tab:case1_error_online}. A significant reduction in both the $L^2$ and energy errors is observed during the first two enrichment iterations, indicating the effectiveness of the online basis functions in capturing the remaining approximation error.

\begin{figure}[htbp]
\centering
\includegraphics[scale=0.4]{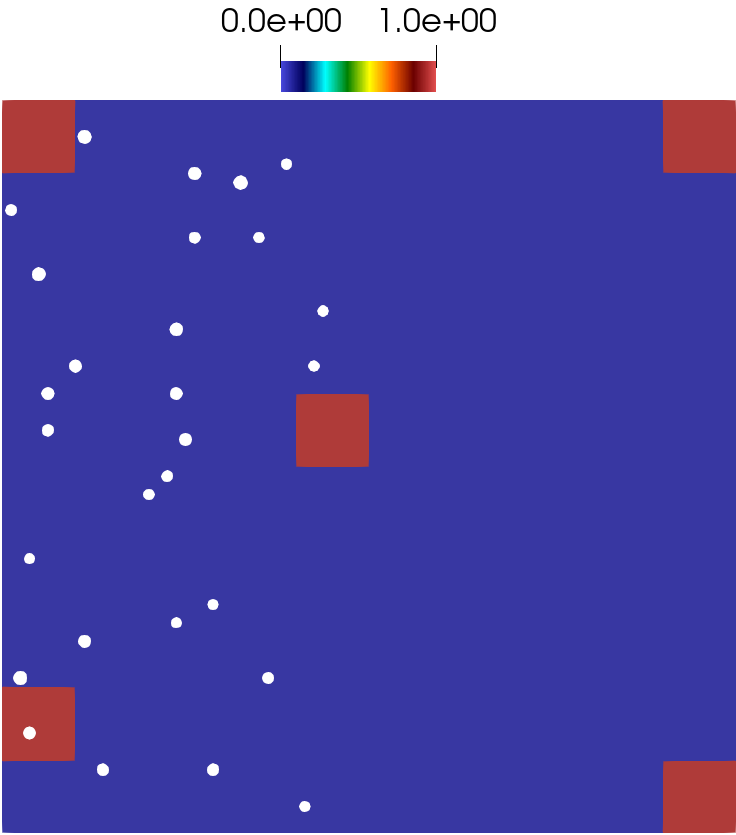}
\includegraphics[scale=0.4]{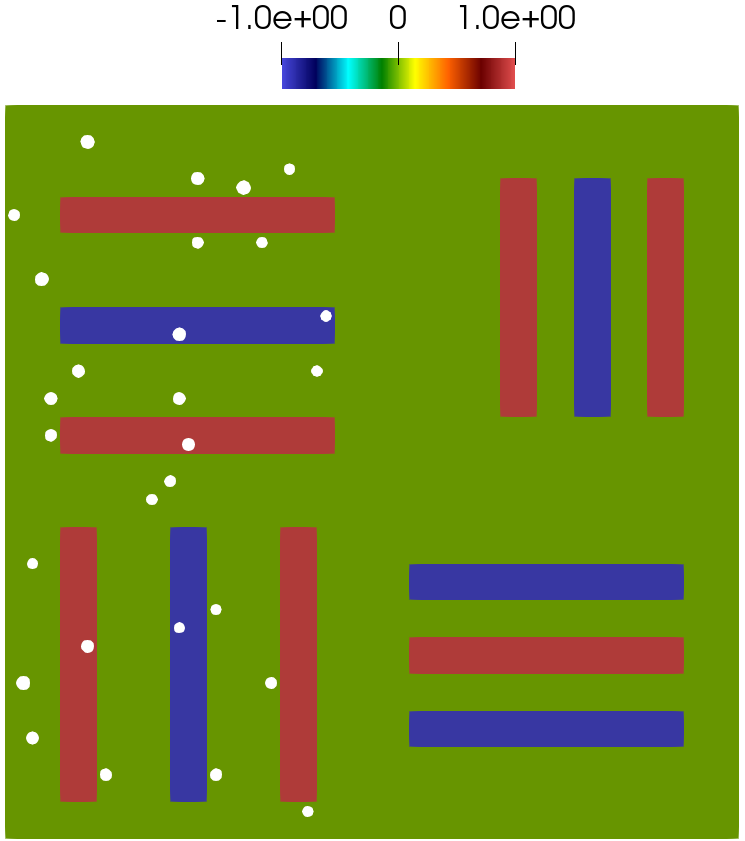}
\caption{Source terms in Case 1. Left: first component. Right: second component.}
\label{fig:source_case1}
\end{figure}

\begin{table}[htbp]
\centering
\begin{tabular}{cccccc}
\hline
\multirow{2}*{$H$} & \multirow{2}*{$m$} & 
\multicolumn{2}{c}{Constraint version} & 
\multicolumn{2}{c}{Relaxed version}
\\ \cline{3-6}
& & $e_{L^2}$ & $e_{H^1}$
& $e_{L^2}$ & $e_{H^1}$ \\ \hline
$1/10$ & 2 & 1.11e-01 & 2.30e-01 & 8.83e-02 & 2.15e-01
\\ \hline
$1/20$ & 4 & 5.66e-03 & 6.89e-02 & 5.51e-03 & 6.83e-02
\\ \hline
$1/40$ & 6 & 1.40e-04 & 5.10e-03 & 1.36e-04 & 5.00e-03
\\ \hline
\end{tabular}
\caption{Numerical errors of CEM-GMsFEM in Case 1 for various coarse mesh sizes $H$.}
\label{tab:case1_error}
\end{table}

\begin{figure}[htbp]
\centering
\includegraphics[scale=0.42]{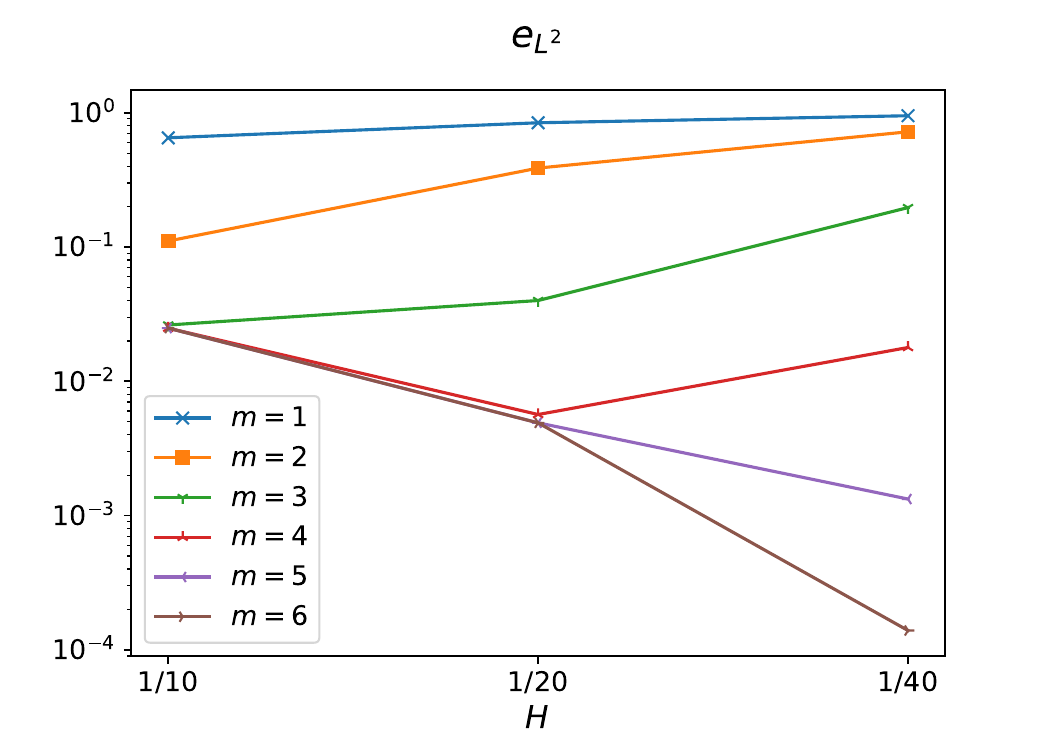}
\includegraphics[scale=0.42]{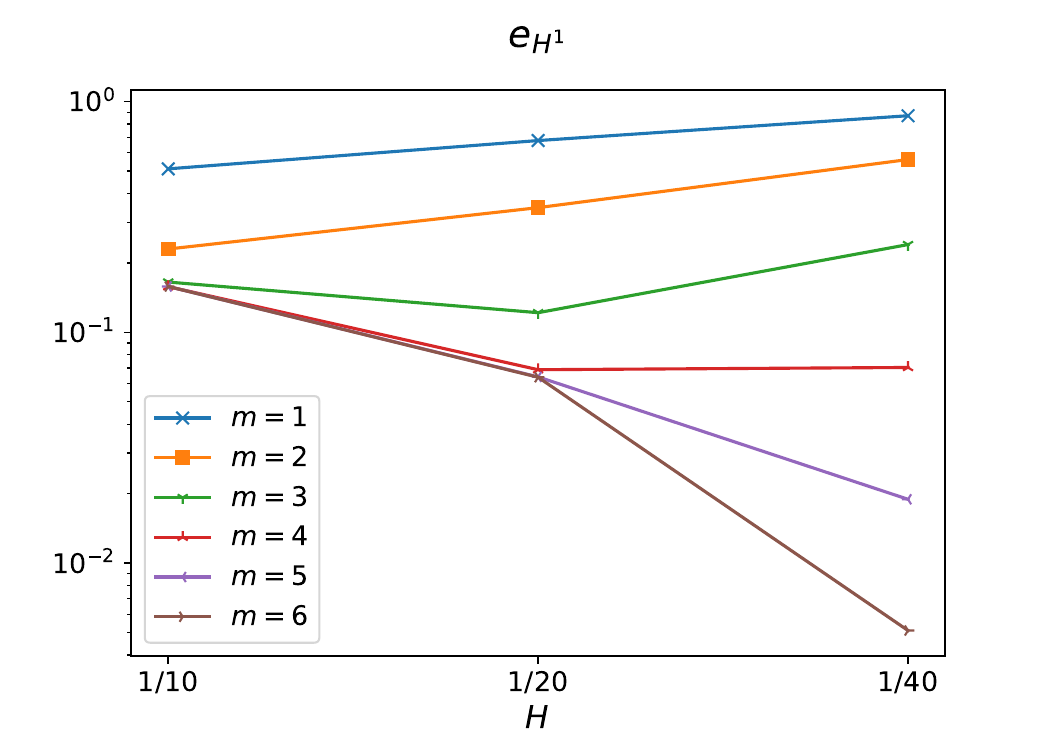}
\caption{Numerical errors obtained using the constraint multiscale basis in the offline stage for various coarse mesh sizes in Case 1.}
\label{fig:case1_error_c}
\end{figure}

\begin{figure}[htbp]
\centering
\includegraphics[scale=0.42]{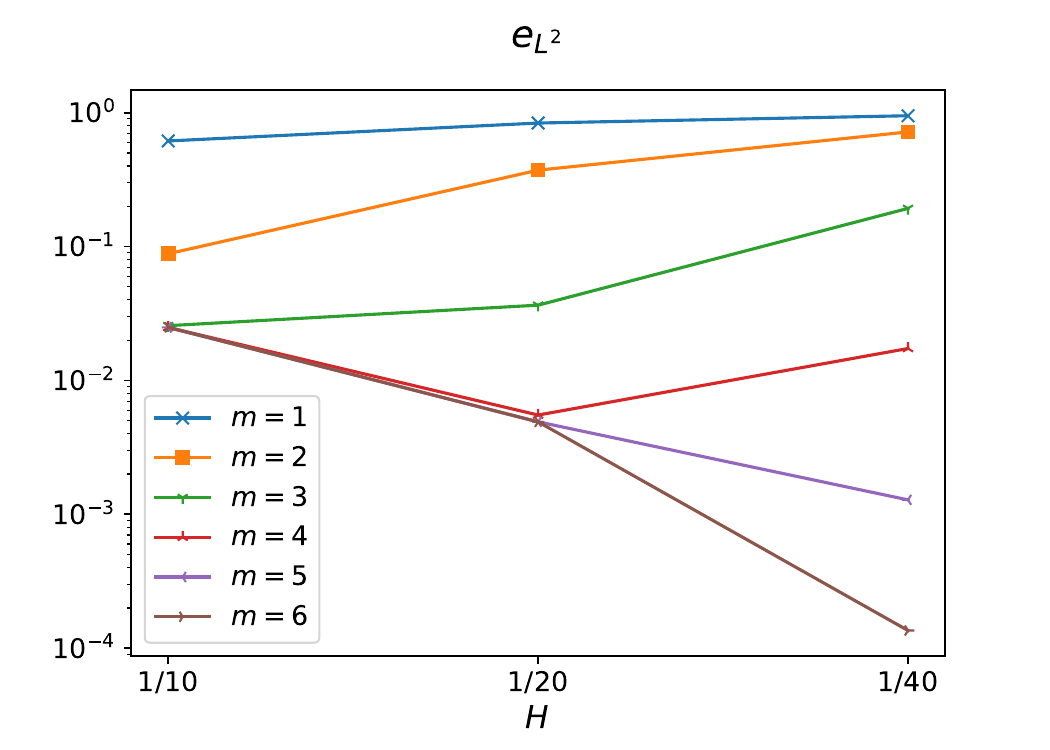}
\includegraphics[scale=0.42]{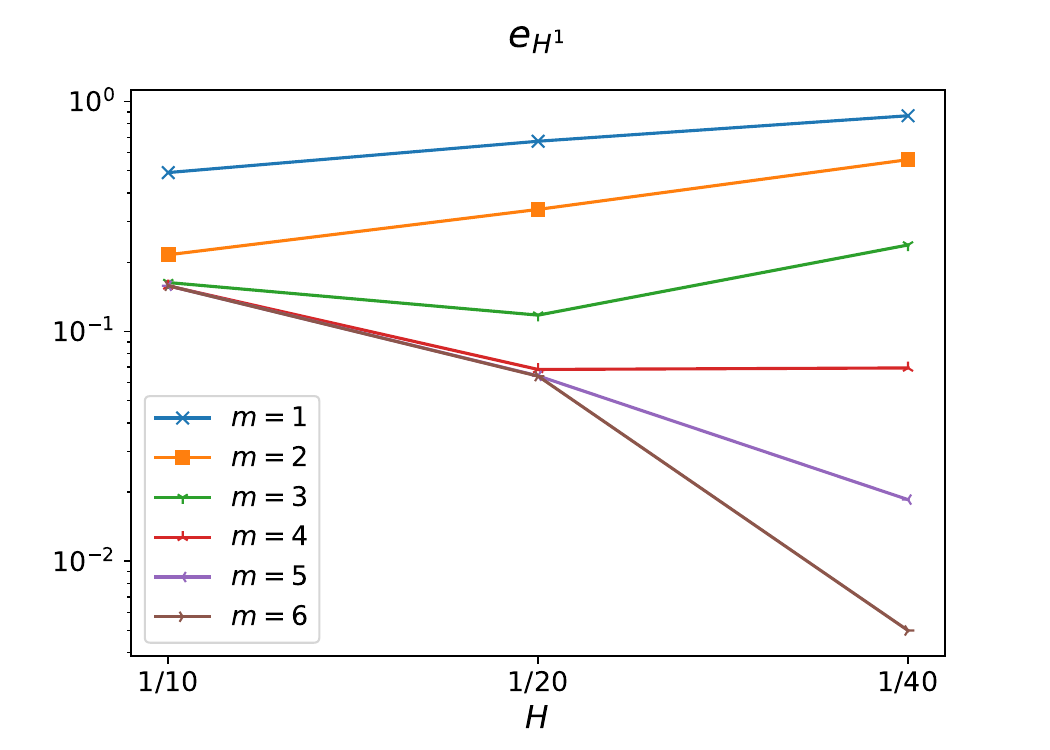}
\caption{Numerical errors obtained using the relaxed multiscale basis in the offline stage for various coarse mesh sizes in Case 1.}
\label{fig:case1_error_r}
\end{figure}

\begin{figure}[htbp]
\centering
\begin{subfigure}[b]{0.29\linewidth}
\includegraphics[width=\linewidth]{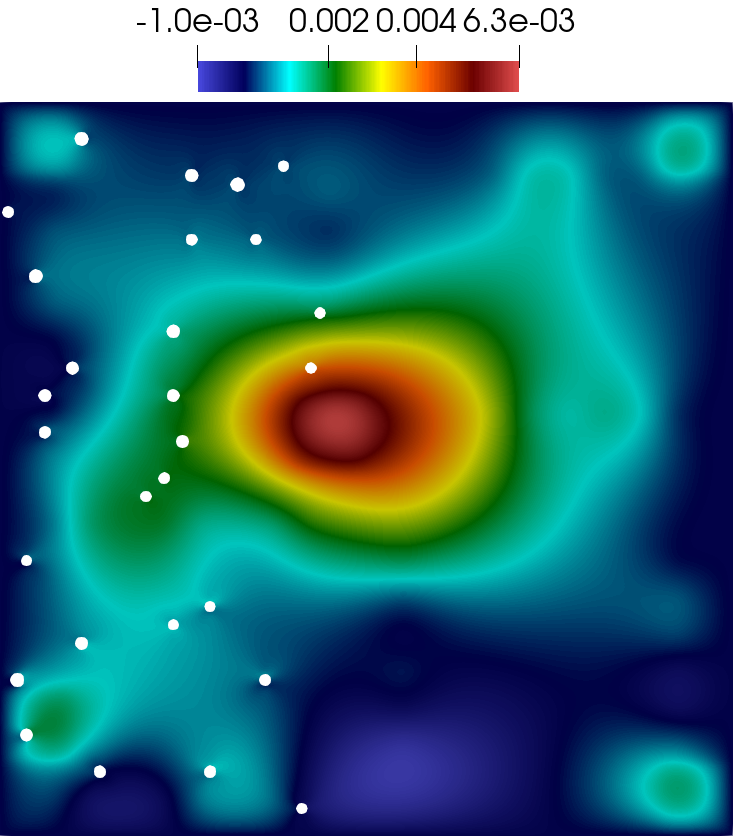}
\caption{}
\end{subfigure}
\begin{subfigure}[b]{0.29\linewidth}
\includegraphics[width=\linewidth]{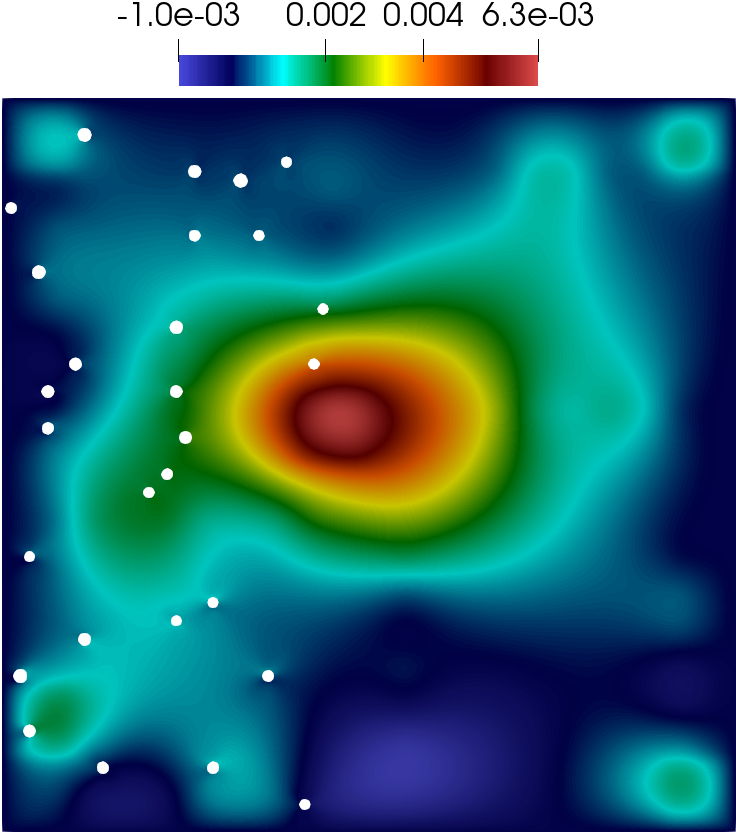}
\caption{}
\end{subfigure}
\begin{subfigure}[b]{0.29\linewidth}
\includegraphics[width=\linewidth]{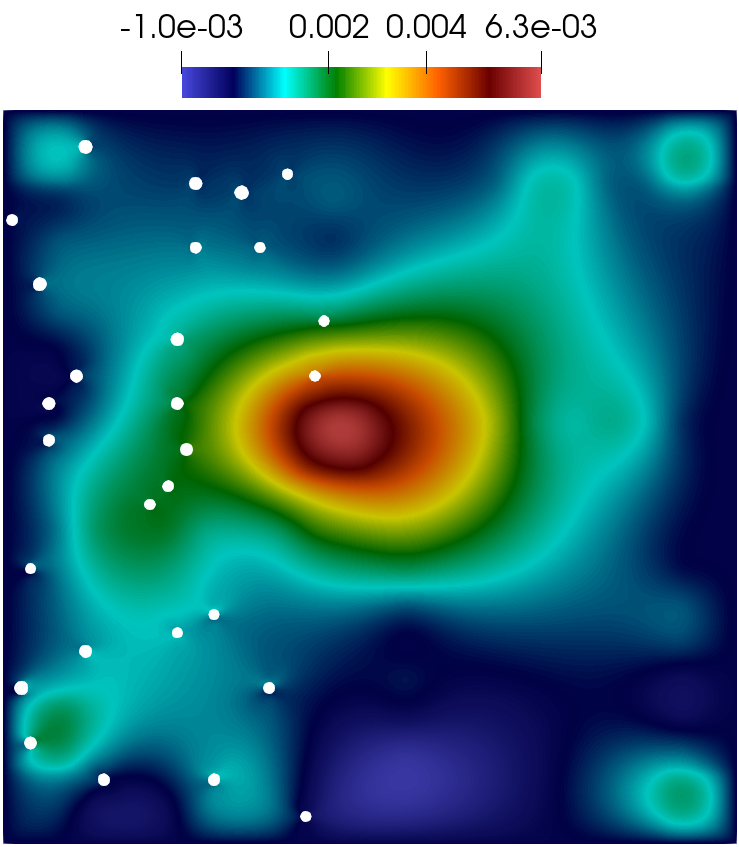}
\caption{}
\end{subfigure}
\\
\begin{subfigure}[b]{0.29\linewidth}
\includegraphics[width=\linewidth]{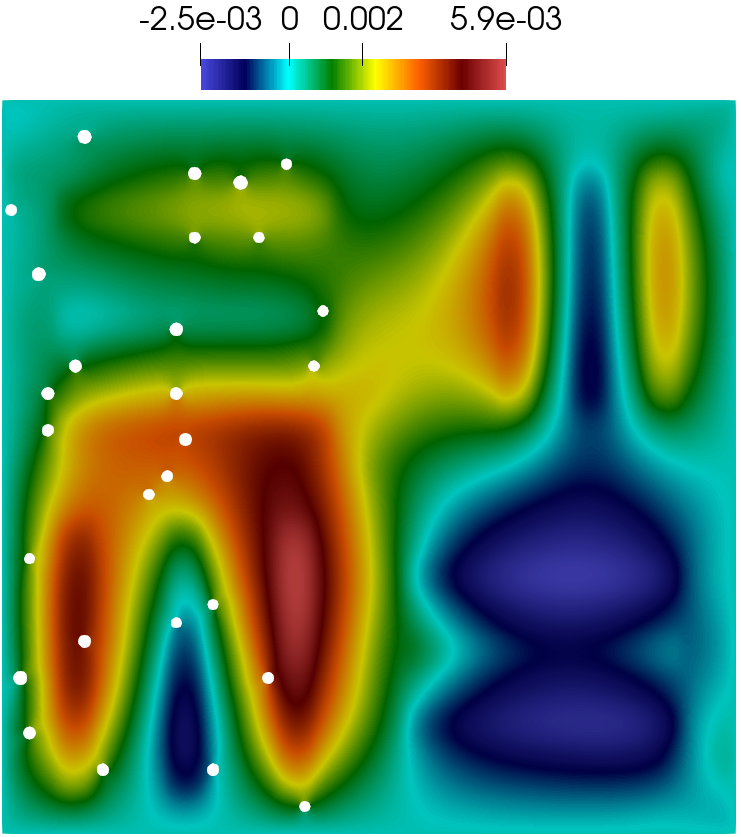}
\caption{}
\end{subfigure}
\begin{subfigure}[b]{0.29\linewidth}
\includegraphics[width=\linewidth]{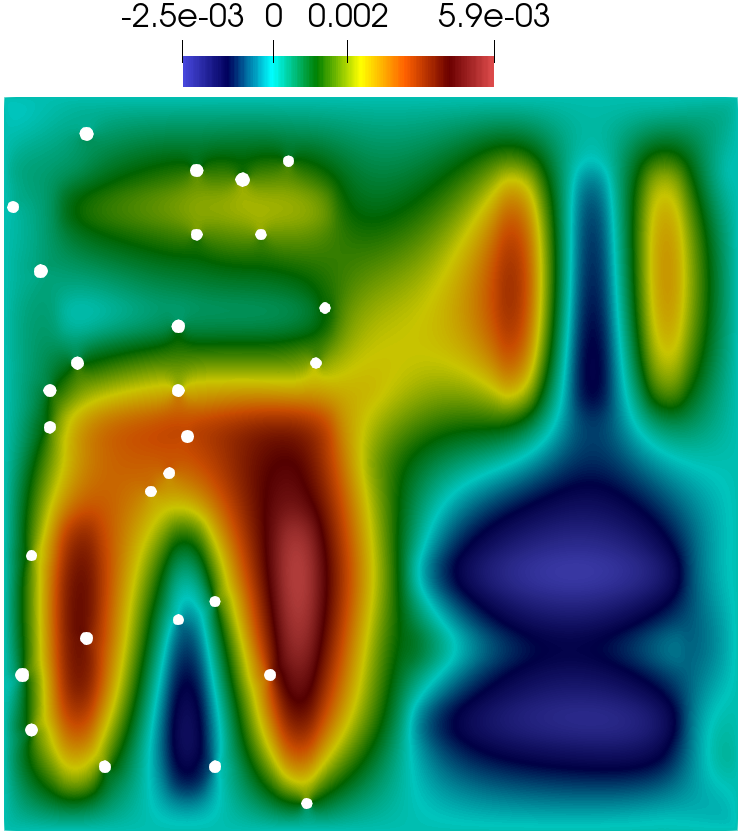}
\caption{}
\end{subfigure}
\begin{subfigure}[b]{0.29\linewidth}
\includegraphics[width=\linewidth]{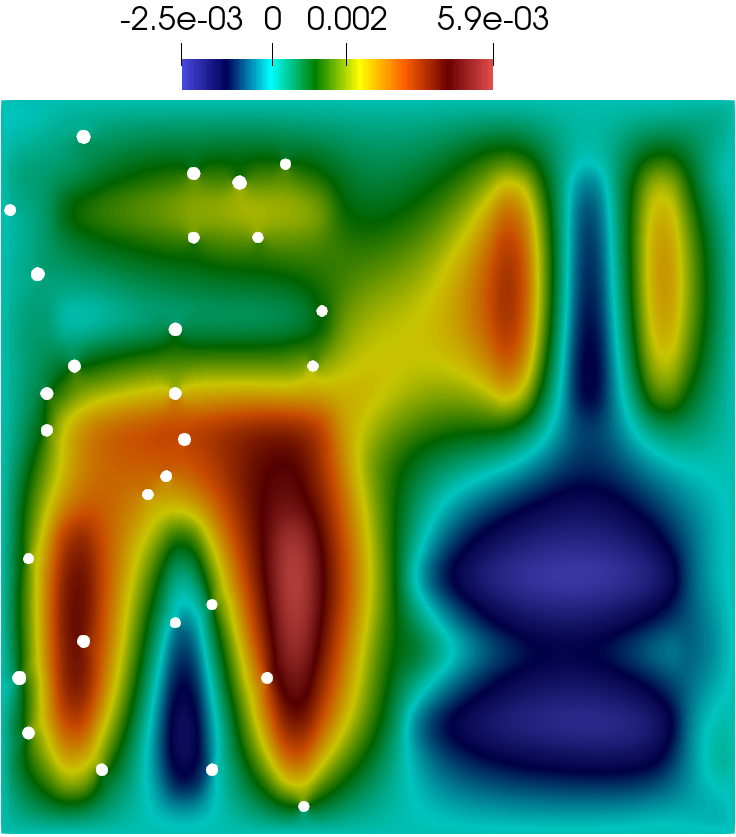}
\caption{}
\end{subfigure}
\caption{Comparison of the reference and offline multiscale solutions in Case 1 with coarse mesh size $H=1/40$ and six oversampling layers. (a)--(c): first component; (d)--(f): second component. In each row, the left panel shows the reference solution, the center panel shows the multiscale solution obtained using the constraint multiscale basis, and the right panel shows the multiscale solution obtained using the relaxed multiscale basis.}
\label{fig:uuhplot_case1}
\end{figure}

\begin{table}[htbp]
\centering
\begin{tabular}{ccccccc}
\hline
\multirow{2}*{\texttt{Iter}} & 
\multicolumn{3}{c}{Uniform enrichment ($\theta=1$)} & 
\multicolumn{3}{c}{Adaptive enrichment ($\theta=0.1$)}
\\ \cline{2-7}
& DOF & $e_{L^2}$ & $e_{H^1}$
& DOF & $e_{L^2}$ & $e_{H^1}$ \\ \hline
0 & 1600 & 3.65e-02 & 1.18e-01 & 1600 & 3.65e-02 & 1.18e-01 \\
1 & 1961 & 5.68e-04 & 6.61e-03 & 1825 & 3.86e-03 & 3.77e-02 \\
2 & 2322 & 9.68e-05 & 2.20e-03 & 2024 & 6.20e-04 & 9.39e-03 \\
3 & 2683 & 4.43e-05 & 1.15e-03 & 2242 & 1.60e-04 & 3.20e-03 \\
4 & 3044 & 1.99e-05 & 6.37e-04 & 2459 & 6.03e-05 & 1.41e-03
\\ \hline
\end{tabular}
\caption{Online error decay history for Case 1.}
\label{tab:case1_error_online}
\end{table}

\subsection{Case 2}

In this example, we consider the second perforated medium shown in Figure \ref{fig:media}, which contains a significantly larger number of perforations than the previous case. The corresponding source term is presented in Figure \ref{fig:source_case2}.

For the offline stage, we investigate the effects of the coarse mesh size and the number of oversampling layers on the approximation accuracy. The convergence results for the constraint and relaxed multiscale bases are shown in Figures \ref{fig:case2_error_c} and \ref{fig:case2_error_r}, respectively. As expected, increasing the number of oversampling layers improves the accuracy of the multiscale approximation. Moreover, reducing the coarse mesh size $H$ together with sufficient oversampling leads to a substantial decrease in the error, as illustrated in Table \ref{tab:case2_error}. Similar to the previous example, the accuracy is ultimately limited by the coarse-grid approximation error. Comparisons between the reference solution and the multiscale solution are presented in Figure \ref{fig:uuhplot_case2}, demonstrating that the proposed method accurately captures the fine-scale features of the solution.

We further investigate the influence of the number of auxiliary basis functions. Fixing $H=1/20$ and $k=4$, the corresponding errors for the constraint and relaxed multiscale bases are reported in Figure \ref{fig:case2_error_li}. The results show that the accuracy improves as more basis functions are included, which is consistent with the theoretical analysis.

Finally, we use this medium to assess the performance of the adaptive online enrichment algorithm \ref{alg:cem_online_adap}. Using the same coarse mesh size and oversampling parameter as in the previous example ($H=1/20$ and $m=3$), together with four basis functions in each coarse block, we set $\theta=0.7$ and $\texttt{Iter}=5$. The corresponding results are reported in Table \ref{tab:case2_error_online}. A significant error reduction is achieved during the first two enrichment iterations, demonstrating the effectiveness of the online basis functions in improving the multiscale approximation.

\begin{figure}[htbp]
\centering
\includegraphics[scale=0.4]{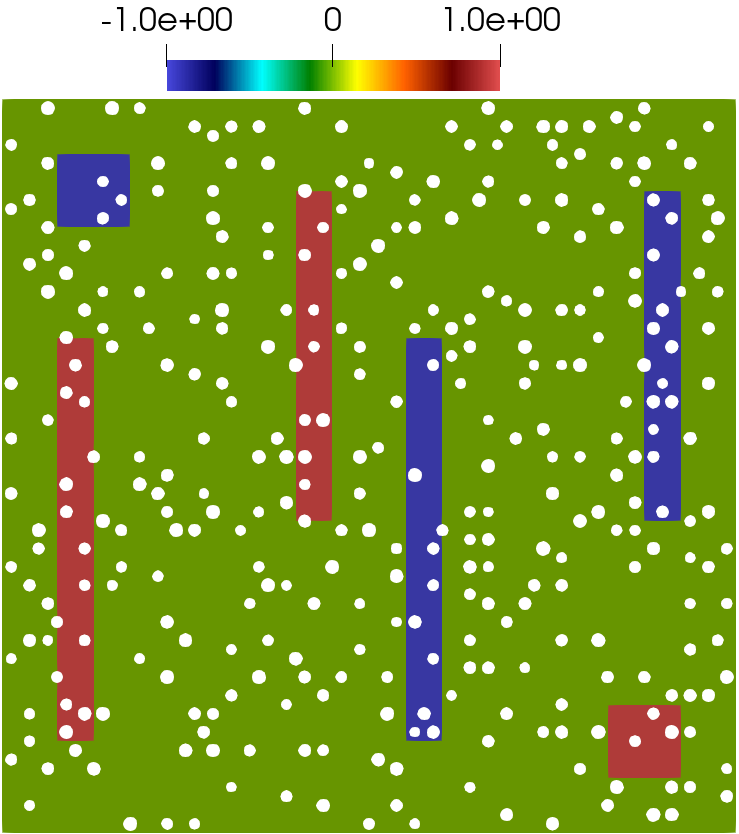}
\includegraphics[scale=0.4]{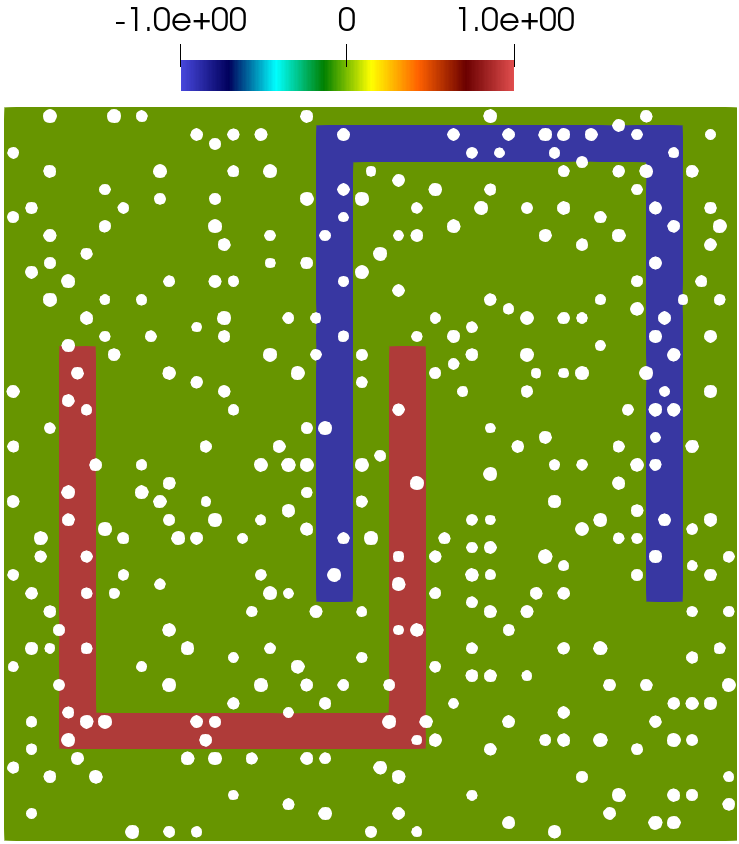}
\caption{Source terms in Case 2. Left: first component. Right: second component.}
\label{fig:source_case2}
\end{figure}

\begin{table}[htbp]
\centering
\begin{tabular}{cccccc}
\hline
\multirow{2}*{$H$} & \multirow{2}*{$m$} & 
\multicolumn{2}{c}{Constraint version} & 
\multicolumn{2}{c}{Relaxed version}
\\ \cline{3-6}
& & $e_{L^2}$ & $e_{H^1}$
& $e_{L^2}$ & $e_{H^1}$ \\ \hline
$1/10$ & 2 & 8.80e-02 & 2.38e-01 & 7.21e-02 & 2.25e-01
\\ \hline
$1/20$ & 4 & 7.13e-03 & 8.01e-02 & 7.11e-03 & 7.99e-02
\\ \hline
$1/40$ & 6 & 4.26e-04 & 1.16e-02 & 4.45e-04 & 1.19e-02
\\ \hline
\end{tabular}
\caption{Numerical errors of CEM-GMsFEM in Case 2 for various coarse mesh sizes $H$.}
\label{tab:case2_error}
\end{table}

\begin{figure}[htbp]
\centering
\includegraphics[scale=0.42]{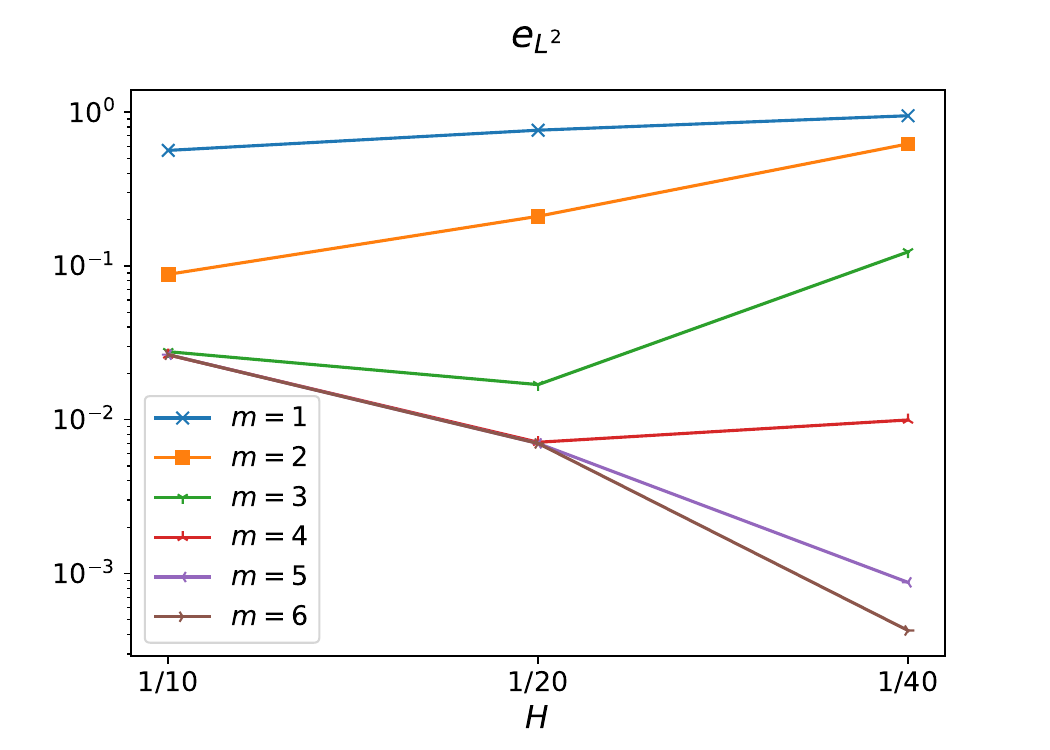}
\includegraphics[scale=0.42]{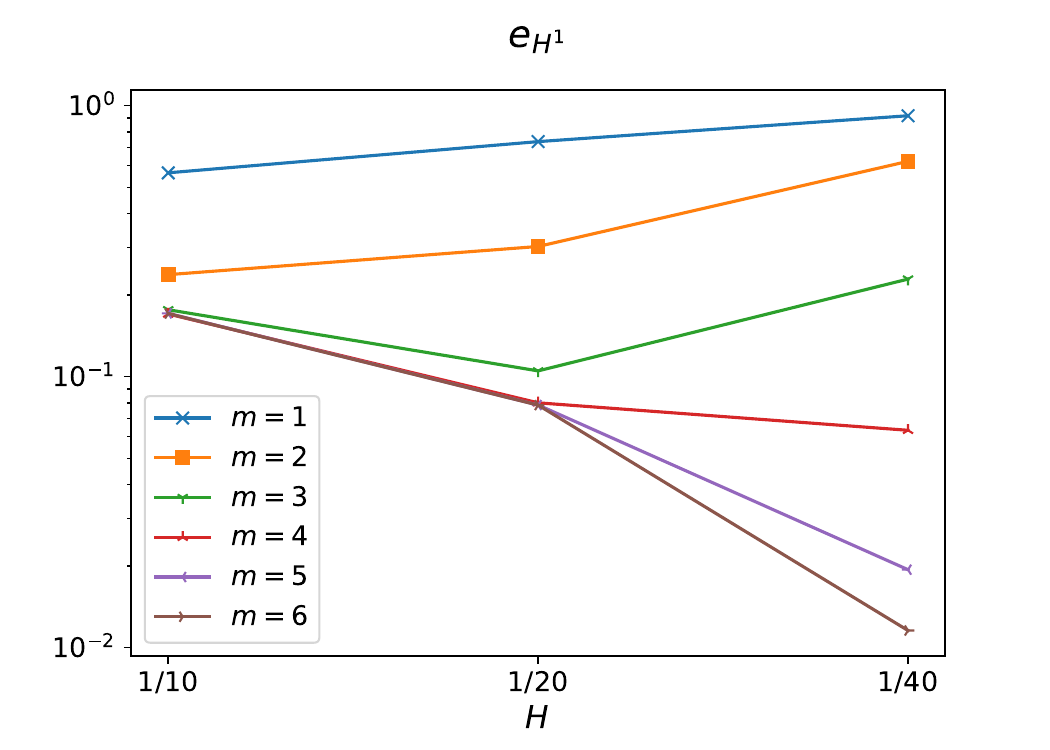}
\caption{Numerical errors obtained using the constraint multiscale basis in the offline stage for various coarse mesh sizes in Case 2.}
\label{fig:case2_error_c}
\end{figure}

\begin{figure}[htbp]
\centering
\includegraphics[scale=0.42]{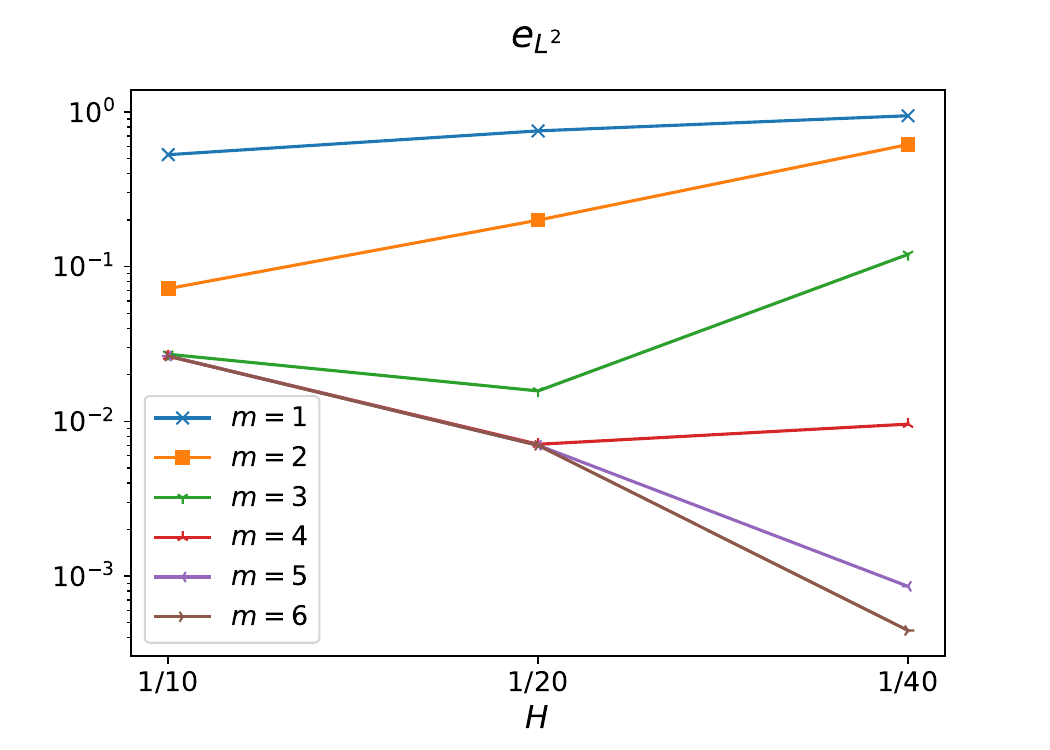}
\includegraphics[scale=0.42]{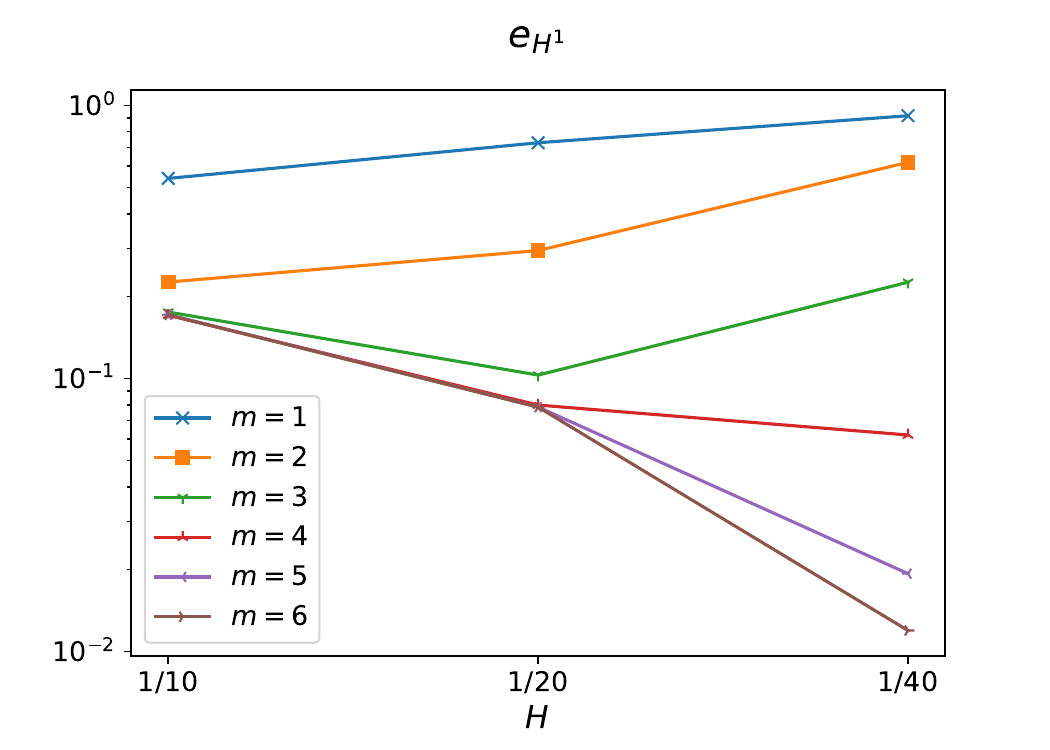}
\caption{Numerical errors obtained using the relaxed multiscale basis in the offline stage for various coarse mesh sizes in Case 2.}
\label{fig:case2_error_r}
\end{figure}

\begin{figure}[htbp]
\centering
\begin{subfigure}[b]{0.29\linewidth}
\includegraphics[width=\linewidth]{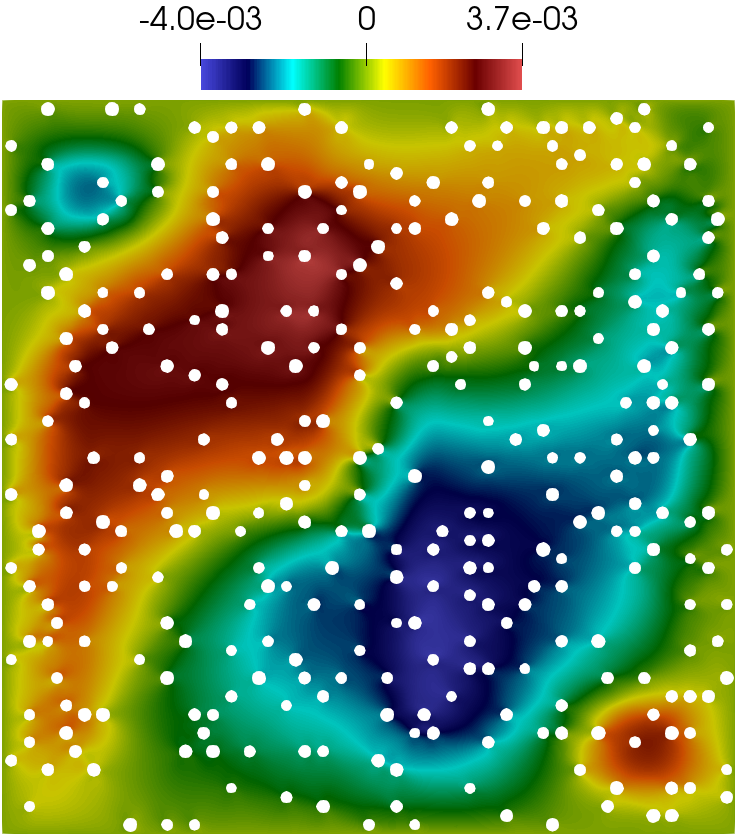}
\caption{}
\end{subfigure}
\begin{subfigure}[b]{0.29\linewidth}
\includegraphics[width=\linewidth]{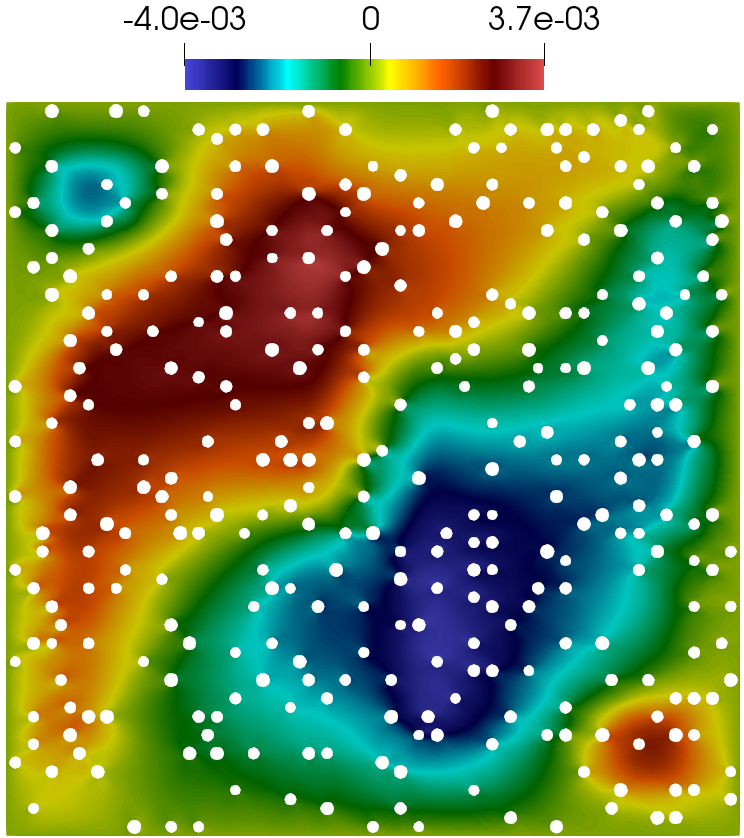}
\caption{}
\end{subfigure}
\begin{subfigure}[b]{0.29\linewidth}
\includegraphics[width=\linewidth]{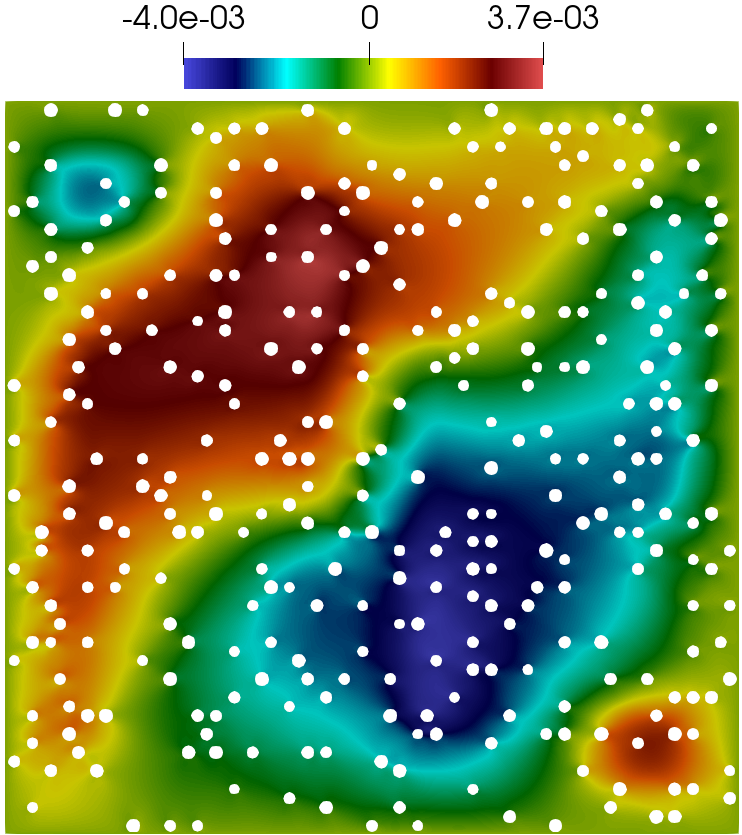}
\caption{}
\end{subfigure}
\\
\begin{subfigure}[b]{0.29\linewidth}
\includegraphics[width=\linewidth]{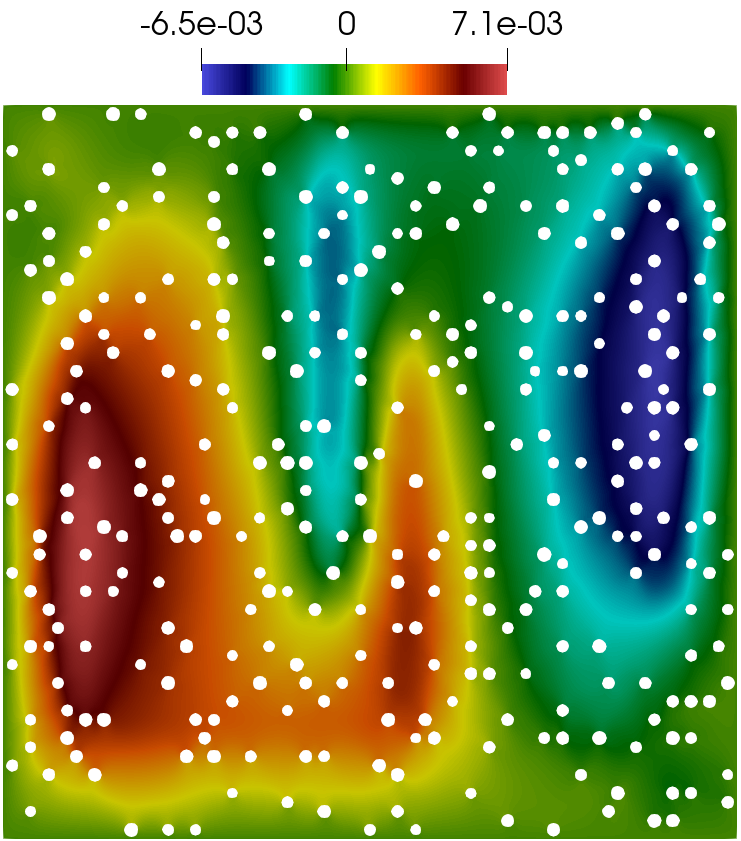}
\caption{}
\end{subfigure}
\begin{subfigure}[b]{0.29\linewidth}
\includegraphics[width=\linewidth]{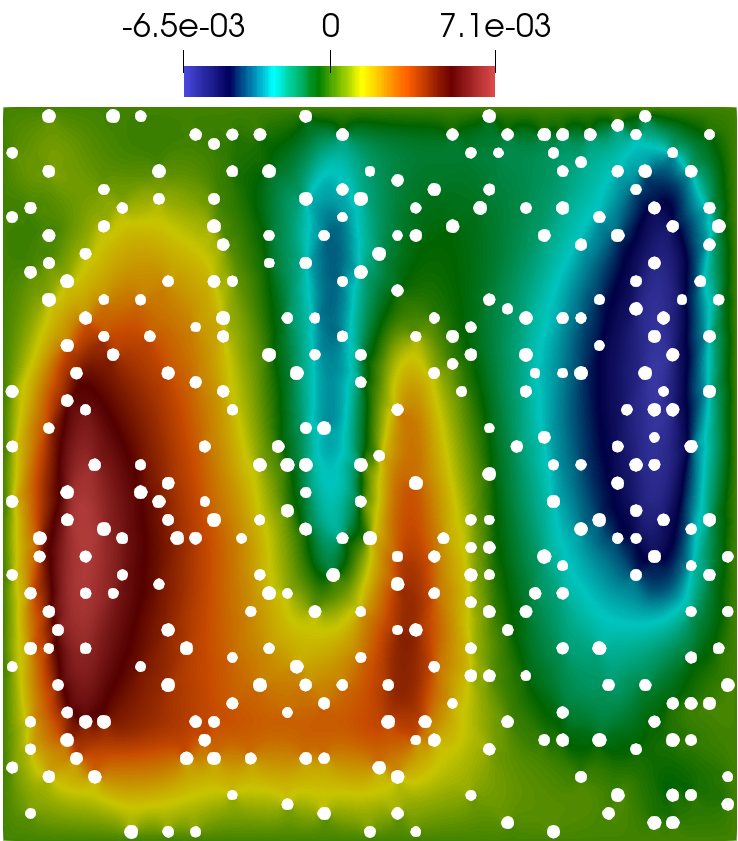}
\caption{}
\end{subfigure}
\begin{subfigure}[b]{0.29\linewidth}
\includegraphics[width=\linewidth]{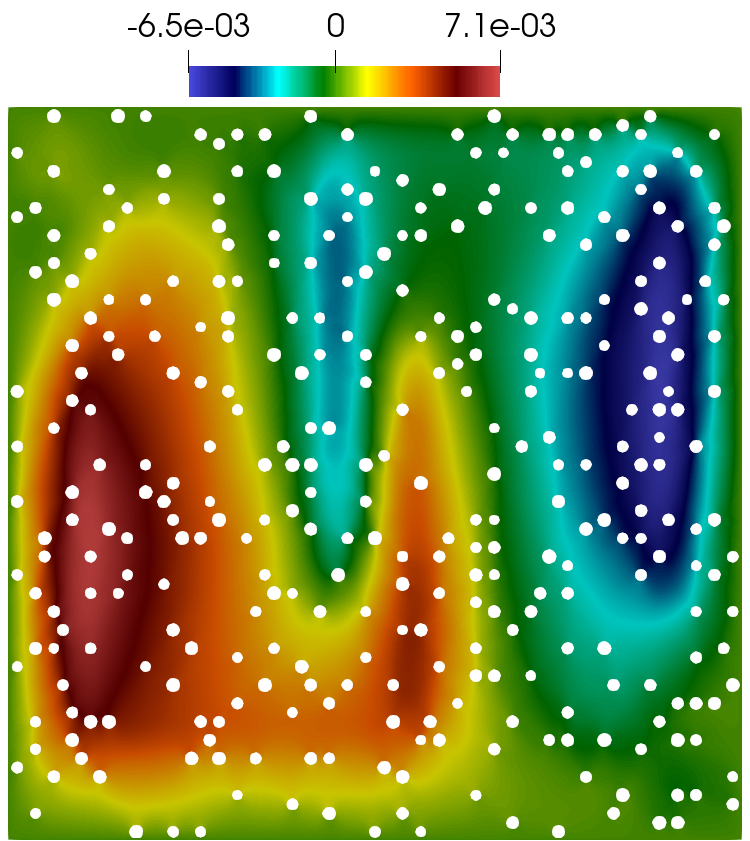}
\caption{}
\end{subfigure}
\caption{Comparison of the reference and offline multiscale solutions in Case 2 with coarse mesh size $H=1/40$ and six oversampling layers. (a)--(c): first component; (d)--(f): second component. In each row, the left panel shows the reference solution, the center panel shows the multiscale solution obtained using the constraint multiscale basis, and the right panel shows the multiscale solution obtained using the relaxed multiscale basis.}
\label{fig:uuhplot_case2}
\end{figure}

\begin{figure}[htbp]
\centering
\includegraphics[scale=0.4]{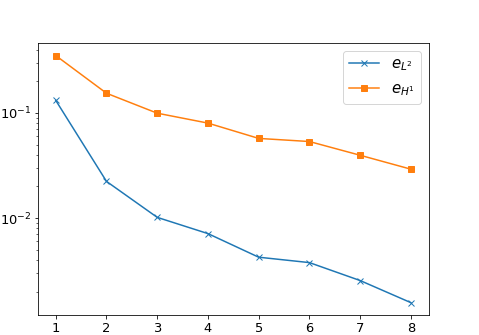}
\includegraphics[scale=0.4]{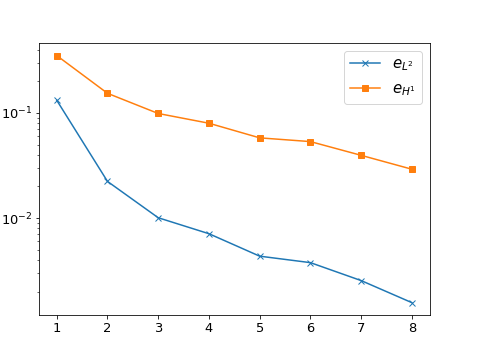}
\caption{Numerical errors for different numbers of offline basis functions in Case 2.}
\label{fig:case2_error_li}
\end{figure}

\begin{table}[htbp]
\centering
\begin{tabular}{ccccccc}
\hline
\multirow{2}*{\texttt{Iter}} & 
\multicolumn{3}{c}{Uniform enrichment} ($\theta=1$)& 
\multicolumn{3}{c}{Adaptive enrichment ($\theta=0.7$)}
\\ \cline{2-7}
& DOF & $e_{L^2}$ & $e_{H^1}$
& DOF & $e_{L^2}$ & $e_{H^1}$ \\ \hline
0 & 1600 & 1.57e-02 & 1.03e-01 & 1600 & 1.57e-02 & 1.03e-01 \\
1 & 1961 & 5.42e-04 & 9.44e-03 & 1641 & 1.09e-02 & 8.03e-02 \\
2 & 2322 & 1.69e-05 & 3.75e-03 & 1695 & 7.16e-03 & 6.58e-02 \\
3 & 2683 & 6.99e-05 & 1.90e-03 & 1750 & 5.28e-03 & 5.17e-02 \\
4 & 3044 & 3.23e-05 & 1.05e-04 & 1802 & 3.74e-03 & 4.18e-02 
\\ \hline
\end{tabular}
\caption{Online error decay history for Case 2.}
\label{tab:case2_error_online}
\end{table}

\section{Conclusions} \label{sec:conclusions}

In this paper, we developed a CEM-GMsFEM framework for the efficient simulation of linear elasticity problems in heterogeneous perforated domains. The proposed method employs an offline--online decomposition. In the offline stage, localized multiscale basis functions are constructed using only the differential operator and the geometric information of the perforated medium. In the online stage, residual-driven basis functions are adaptively generated to incorporate source-term information and further improve the accuracy of the multiscale approximation.
We established convergence results for both the offline and online stages. In particular, the analysis shows that the localization errors of the multiscale basis functions decay exponentially with respect to the oversampling layers. Moreover, we proved that the oversampling regions required for the online basis functions can be determined using only local geometric information, which reduces the computational cost while preserving the convergence properties of the method.
Numerical experiments on two perforated media with different geometric characteristics demonstrate the accuracy and efficiency of the proposed approach. The results confirm that the offline multiscale space provides a reliable coarse-scale approximation and that the online enrichment procedure can significantly accelerate convergence with only a small number of additional basis functions.
The online basis functions are driven by global (residual) information, which enables rapid error reduction without further mesh refinement. They are therefore particularly beneficial when higher accuracy is desired without refining the coarse mesh size. In addition, the online procedure is expected to be especially useful for nonlinear problems and other challenging multiscale settings where local offline spaces alone may be insufficient to capture solution-dependent features, which will be investigated in future work.

\section*{Acknowledgements}
Wei Xie and Yin Yang are supported by the National Natural Science Foundation of China Project (No. 12571469), Scientific Research Innovation Capability Support Project for Young Faculty of China (No. SRICSPYF-BS2025132), the Project of Scientific Research Fund of the Hunan Provincial Science and Technology Department (No. 2024JJ1008).
Eric Chung is partially supported by the Hong Kong RGC General Research Fund (Projects: 14305423 and 14305624).
Yunqing Huang is partially supported by the 111 Project (No. D23017), Major Scientific and Technological Innovation Plat form Project of Hunan Province (2024JC1003), and Program for Science and Technology Innovative Research Team in Higher Educational Institutions of Hunan Province of China.

\bibliographystyle{abbrv}
\bibliography{refs}
\end{document}